\documentclass[A4,reqno]{amsart}

\usepackage{amsmath,amsthm,amsfonts,amssymb,graphicx,graphics,color}
\def \CE {{ \mathcal{E}}}

\def \Fd   {\CF_{d} }
\def \ot {\otimes}
\def\C{\mathbb{C}}

\def \CH {{ \mathcal{H}}}
\def \CB {{ \mathcal{B}}}

\font\ita=cmssi10  
scaled\magstep2 \font\normal=cmr10

\def\sig{\sigma}
\def\LL{\big\langle}
\def\RR {\big\rangle}
\def\OM {\Omega}

\def\la {\lambda}
\def \sas {\vskip .06truein}
\def\sa{{\vskip .125truein}}

\def \eee {\epsilon}
\def\aaa {\alpha}
\def\aa {\alpha}
\def\ttt{\theta}

\def\dd{\delta}

\def\con {\subseteq}
\def \ses {\enskip = \enskip}
\def \sps {\, + \,}

\def \sms {\, - \,}

\def \scs {\, , \,}
\def \ess {\enskip}

\def \ssp {\hskip .25em}
\def \bigsp {\hskip .5truein}
\def \part {\vdash}

\normal

\date{September 25, 2008}

\def \RA {{ \rightarrow }}
\def \CC {{ \mathcal{C}}}
\def \CF {{ \mathcal{F}}}
\def \CS {{ \mathcal{S}}}
\def\xon {x_1,x_2,\ldots ,x_n}
\def \om {\omega}

\def \scs {\ssp , \ssp}
\def \ess {\enskip}
\def \ssp {\hskip .25em}
\def \bigsp {\hskip .5truein}
\def \part {\vdash}
\font\normal=cmr10
\def \I {{\rm I}}
\newcommand\mymatrix[1]{\begin{matrix}
  #1
\end{matrix}}
\newcommand\eqalign[1]{\begin{matrix}
  #1
\end{matrix} }

\def\mytopsep{3mm}

\newtheoremstyle{myremark}{\mytopsep}{\mytopsep}{\normalfont}{0pt}{\bfseries}{.}{3mm}{}

\newtheorem{Theorem}{Theorem}[section]
\newtheorem{Proposition}[Theorem]{Proposition}

\theoremstyle{myremark}
\newtheorem{Remark}[Theorem]{Remark}
\newtheorem{Example}[Theorem]{Example}
\newtheorem{Algorithm}[Theorem]{Algorithm}

\newcommand\myitem[1]{\item[#1]}

\def\multi#1{\vbox{\baselineskip=0pt
\halign{\hfil$\scriptstyle\vphantom{(_)}##$\hfil\cr#1\crcr}}}
\def\mypicture #1 by #2 (#3){
  \vbox to #2{
    \hrule width #1 height 0pt depth 0pt
    \vfill
    \special{picture #3} 
    }
  }

\def \con {\subseteq}

\title[Invariants, Kronecker Products, and Diophantine Systems]{Invariants, Kronecker Products, and \\
Combinatorics of Some Remarkable Diophantine Systems\\ (Extended
Version)}

\author[A. Garsia, G. Musiker, N. Wallach, G. Xin]{Adriano Garsia$^1$, Gregg Musiker$^2$, Nolan Wallach$^3$, Guoce Xin$^4$}

\address{$^{1,2,3}$ Department of Mathematics, UCSD, CA\\
$^2$ Department of Mathematics, MIT, MA\\
 $^4$ Center for
Combinatorics, LPMC-TJKLC,
 Nankai University,
 Tianjin 300071, P.
R. China }

\begin{document}
\maketitle

\begin{abstract}
This work lies across three areas (in the title) of investigation
that are by themselves of independent interest. A problem that arose
in quantum computing led us to a link that tied these areas
together. This link consists of a single formal power series with a
multifaced interpretation. The deeper exploration of this link
yielded results as well as methods for solving some numerical
problems in each of these separate areas.
\end{abstract}

\begin{small}

\end{small}

\def \BF {{\bf F}}
\def \BB {{\bf B}}
\def \BA {{\bf A}}
\def \BM {{\bf M}}
\def \BR {{\bf R}}
\def \I {{\rm I}}

\section{Introduction}

Since our work may be of interest to audiences of varied background
we will try to keep our notation as elementary as possible and
entirely self contained.

The problem in invariant theory that was the point of departure in
our investigation is best stated  in its simplest and most
elementary version. Given two matrices $ A = \Big[\mymatrix{a_{11}&
a_{12} \cr a_{21}& a_{22} \cr }\Big]$ and $B =
\Big[\mymatrix{b_{11}& b_{12} \cr b_{21}& b_{22} \cr }\Big]$ of
determinants $1$, or equivalently in $SL[2]:=SL(2,\C)$, we recall
that their tensor product may be written in the block form
\begin{align}
 A\otimes B= \left[\mymatrix{a_{11} B & a_{12} B\cr a_{21} B &
a_{22} B\cr}\right]. \label{I.1}
\end{align}

 We also recall that the action of a matrix $
M=[m_{ij}]_{i,j=1}^n$ on a polynomial $P(x)$ in $\BR_n:=\C[\xon]$
may be defined  by setting
\begin{align}
T_MP(x)= P(xM), \label{I.2}
\end{align}
where the symbol $xM$ is to be interpreted as multiplication of a
row $n$-vector by an $n\times n$ matrix.
 This given, we denote by $\BR_4^{SL[2]\otimes SL[2]}$ the ring of polynomials in $\BR_4$ that  are invariant
under the action of $ A\otimes B$ for all pairs ${A,B}\in SL[2]$. In
symbols
\begin{align}
 \BR_4^{SL[2]\otimes SL[2]}= \big  \{P\in \BR_4\, :\,
T_{A\otimes B}P(x)=P(x)\, \big  \}. \label{I.3}
\end{align}
Since the action in \eqref{I.2} preserves degree and homogeneity,
$\BR_4^{SL[2]\otimes SL[2]} $ is graded, and as a vector space it
decomposes into the direct sum
$$
\BR_4^{SL[2]\otimes SL[2]}= \bigoplus_{m\ge
0}\CH_m\big(\BR_4^{SL[2]\otimes SL[2]} \big),
$$
where the $m^{\mathrm{th}}$ direct summand  here denotes the
subspace consisting of the   $SL[2]\otimes SL[2]$-invariants that
are homogenous of degree $m$. The natural problem then arises to
determine the Hilbert series
$$
W_2(q)= \sum_{m\ge 0} \, q^m \dim \, \CH_m\big(\BR_4^{SL[2]\otimes
SL[2]} \big).
$$
Now note that using \eqref{I.1} iteratively we can define the
$k$-fold tensor product $ A_1\otimes A_2\otimes\cdots \otimes A_k, $
and thus  extend \eqref{I.3} to its general form
\begin{align*}
\BR_{2^k}^{SL[2]\otimes SL[2] \otimes\cdots \otimes SL[2]}= \big
\{P\in \BR_{2^k}\, :\, T_{A_1\otimes A_2\otimes\cdots
\otimes A_k}P(x)=P(x)\, \big  \} 
\end{align*}
and set
\begin{align*}
 W_k(q)= \sum_{m\ge 0} \, q^m \dim \,
\CH_m\big(\BR_{2^k}^{SL[2]\otimes SL[2] \otimes\cdots \otimes
SL[2]}\ \big). 
\end{align*}
Remarkably, to this date only  the series  $W_2(q), W_3(q), W_4(q),
W_5(q)$ are known explicitly. Moreover, although the  three series
$W_2(q),W_3(q),W_4(q)$ may be hand computed, so far $W_5(q)$ has
only been obtained by computer.

The third named author, using branching tables calculated in
\cite{7}, was able to predict the explicit form of $W_5(q)$ by
computing a sufficient number of its coefficients. The computation
of these tables took  approximately 50 hours using an array of $9$
computers.

The series $ W_4(q), W_5(q)$ first appeared in print in works of
Luque-Thibon \cite{4}, \cite{5} which were motivated by the same
problem of quantum computing. We understand that their computation
of $W_5(q)$ was carried out by a brute force use of the partial
fraction algorithm of the fourth named author, and  it required
several hours with the computers of that time.

The present work was carried out whilst unaware of the work of
Luque-Thibon. Our main  goal is to acquire a theoretical
understanding of the combinatorics underlying  such Hilbert series
and give a more direct construction of $W_5(q)$ and perhaps bring
$W_6(q)$ within reach of present computers.

 Fortunately, as is often the case with a difficult problem, the methods that
are developed to solve it  may be more significant than the  problem
itself. This is no exception as we shall see.

Let us recall that  the pointwise product of two characters
$\chi^{(1)}$ and $\chi^{(2)}$ of the symmetric group $S_n$ is also a
character of $S_n$, and we shall denote it here by
$\chi^{(1)}\odot\chi^{(2)}$. This is usually called the \hbox{\emph{
Kronecker }} product of $\chi^{(1)}$ and $\chi^{(2)}$. An
outstanding yet unsolved problem is to obtain a combinatorial rule
for the computation of the integer
\begin{align}
 c_{\la^{(1)},\la^{(2)},\ldots ,\la^{(k)}}^\la \label{I.6}
\end{align} giving the multiplicity of $\chi^\la$ in the Kronecker product
$\chi^{\la^{(1)}}\odot \chi^{\la^{(2)}}\odot\cdots \odot
\chi^{\la^{(k)}}$. Here $\chi^\la$ and each  $\chi^{\la^{(i)}}$ are
irreducible Young characters of $S_n$. Using the Frobenius map $\bf
F$ that sends the irreducible character $\chi^\la$ onto the Schur
function $S_\la$, we can define the Kronecker product of two
homogeneous symmetric functions of the same degree $f$ and $g$  by
setting
$$
f\odot g = \BF \big((\BF^{-1}f)\odot(\BF^{-1}g) \big).
$$
With this notation the coefficient in \eqref{I.6} may also be
written in the form
\begin{align*}
 c_{\la^{(1)},\la^{(2)},\ldots ,\la^{(k)}}^\la = \LL
s_{\la^{(1)} }\odot s_{\la^{(2)} }\odot\cdots \odot s_{\la^{(k)} }
\scs s_\la\RR, 
\end{align*}
where $\LL\scs \RR$ denotes the customary Hall scalar product of
symmetric polynomials. The relevancy of all this to the previous
problem is a consequence of the following identity. \sas

\begin{Theorem}
\label{t-I.1}
\begin{align}
 W_k(q)= \sum_{d\ge 0} q^{2d}  \LL s_{d,d} \odot s_{d,d} \odot\cdots \odot s_{d,d} \scs s_{2d}\RR
\label{I.8}
\end{align}
where, in each term,  the Kronecker product has $k$ factors.
\end{Theorem}

For this reason, we will often refer to the task of constructing
$W_k(q)$ as the \emph{ Sdd Problem}. Using this connection and some
auxiliary results on the Kronecker product of symmetric functions we
derived in \cite{2} that
\begin{align}
 W_2( {q})= {1\over 1-q^2}\scs \ess\ess\ess W_3(q)= {1\over
1-q^4} \scs \ess\ess\ess  W_4(q)= {1\over (1-q^2)(1-q^4)^2(1-q^6)}.
\label{3.20}
\end{align}
Although  this approach is worth pursuing (see \cite{2}),  the
present investigation  led us to another surprising facet of this
problem.

Let us start with a special  case. We are asked to place integer
weights on the vertices of the unit square so that all the sides
have equal weights. Denoting by $P_{00}$, $P_{01}$, $P_{10}$,
$P_{11}$ the vertices (see figure) and by $p_{00}$, $p_{01}$,
$p_{10}$, $p_{11}$ their corresponding weights, we are led to the
following Diophantine system.

\medskip
\begin{align*}
\CS_2:\bigg\|\eqalign{ & p_{00}+p_{01}\sms p_{10}-p_{11}= 0 \cr
&p_{00}-p_{01}\sps p_{10}-p_{11}=0 \cr }.
\end{align*}

\vskip -.8in\hfill \mbox{\begin{picture}(50,70) \linethickness{1pt}
\put(0,0){\circle{5}}\put(0,40){\circle{5}}\put(40,0){\circle{5}}\put(40,40){\circle{5}}
\put(2.5,0.75){\line(10,0){34.5}}\put(39.5,3.25){\line(0,10){34.5}}\put(0,3.25){\line(0,10){34.5}}
\put(2.5,40.5){\line(10,0){34.5}}
\put(-12.5,-12.5){\footnotesize{$\mathbf{P_{00}}$}}
\put(45.5,-12.5){\footnotesize{$\mathbf{P_{10}}$}}
\put(-12.5,50.5){\footnotesize{$\mathbf{P_{01}}$}}
\put(45.5,50.5){\footnotesize{$\mathbf{P_{11}}$}}
\end{picture}\ \ \ }

\noindent The general solution to this problem may be expressed as
the formal series
$$
F_2(y_{00},y_{01},y_{10},y_{11})= \sum_{p\in \CS_2} y_{00}^{p_{00}}
y_{01}^{p_{01}} y_{10}^{p_{10}} y_{11}^{p_{11}}= {1\over
(1-y_{00}y_{11})(1-y_{01}y_{10})}.
$$
In particular, making the  substitution  $ y_{00}= y_{01}= y_{10}=
y_{11}=q$  we derive that the enumerator of solutions by total
weight is given by  the generating function
$$
G_2(q)= \sum_{d\ge 0} m_d(2) q^{2d} = {1\over (1-q^2)^2},
$$
with $m_d(2)$ giving the number of solutions of total weight $2d$.

This problem generalizes to arbitrary dimensions. That is we seek to
enumerate the distinct ways of placing weights on the vertices
of the unit $k$-dimensional hypercube so that all hyperfaces have the same weight.
Denoting by $p_{\eee_1\eee_2\cdots \eee_k}$ the weight we place on the vertex of
coordinates $(\eee_1,\eee_2,\ldots,  \eee_k)$ we obtain a Diophantine system $\CS_k$ of $k$
equations in the $2^k$
variables $\{p_{\eee_1\eee_2\cdots \eee_k}\}_{\eee_i=0,1}$.

For instance, using this notation, for the $3$-dimensional cube we
obtain the system
$$
\CS_3:\left\|\eqalign{ & p_{000}\sps p_{001}\sps p_{010}\sps
p_{011}\sms p_{100}\sms p_{101}\sms  p_{110}\sms p_{111}\ses 0\cr &
p_{000}\sps p_{001}\sms p_{010}\sms p_{011}\sps p_{100}\sps
p_{101}\sms  p_{110}\sms p_{111}\ses 0\cr & p_{000}\sms p_{001}\sps
p_{010}\sms p_{011}\sps p_{100}\sms p_{101}\sps  p_{110}\sms
p_{111}\ses 0\cr }\right. \ \ .
$$
In this case the enumerator of solutions by total weight is
$$
G_3(q)\ses \sum_{d\ge 0} m_d(3) q^{2d} \ses {1-q^8\over
(1-q^2)^4(1-q^4)^2}.
$$

The relevance of all this to the previous problem is a consequence
of the following identity.

\begin{Theorem}
\label{t-I.2} Denoting by $m_d(k)$ the number of solutions of the
system $\CS_k$ of total weight $2d$ and setting
\begin{align}
 G_k(q)= \sum_{d\ge 0} m_d(k)q^{2d}, \label{I.10}
\end{align}
we have
\begin{align*}
 G_k(q)= \sum_{d\ge 0} q^{2d}  \LL h_{d,d} \odot h_{d,d} \odot\cdots \odot h_{d,d} \scs
 S_{2d}\RR,
\end{align*}
where, $h_{d,d}$ denotes the homogenous basis element indexed by the
two part partition $(d,d)$, and  in each term, the Kronecker product
has $k$ factors.
\end{Theorem}
For this reason, we will refer to  the task of constructing the
series $G_k(q)$ as the \emph{ Hdd Problem}.

Theorem \ref{t-I.2} shows that the algorithmic machinery of
Diophantine analysis may be used in the construction of generating
functions of Kronecker coefficients as well as Hilbert series of
ring of invariants. More precisely we are referring here to the
\emph{ constant term methods} of
 MacMahon partition analysis
which have been recently translated into computer software by
Andrews et al. \cite{1} and Xin \cite{9}.

To see what this leads to, we start by noting that using
MacMahon's approach the solutions of $\CS_2$ may be obtained by the
following identity
$$
F_2(y_{00},y_{01},y_{10},y_{11})= \sum_{p_{00}\ge 0}\sum_{p_{01}\ge
0}\sum_{p_{10}\ge 0}\sum_{p_{11}\ge 0} y_{00}^{p_{00}}
y_{01}^{p_{01}} y_{10}^{p_{10}} y_{11}^{p_{11}}
a_1^{p_{00}+p_{01}\sms p_{10}-p_{11}} a_2^{p_{00}-p_{01}\sps
p_{10}-p_{11}}\Big|_{a_1^0a_2^0},
$$
where the symbol ``$\big|_{a_1^0a_2^0}$" denotes the operator of
taking the constant term in $a_1,a_2$. This identity may also be
written in the form
$$
F_2(y_{00},y_{01},y_{10},y_{11})= {1\over (1-  y_{00}a_1a_2)
(1-y_{01}a_1 /a_2) (1- y_{10}a_2 /a_1) (1- y_{11}/a_1 a_2)}
\Big|_{a_1^0a_2^0}.
$$
In particular the enumerator of the solutions of $\CS_2$ by total weight may be computed from the identity
$$
G_2(q)= {1\over (1- qa_1 a_2)(1- qa_1 /a_2)(1- qa_2/ a_1)(1- q/a_1
a_2)}\Big|_{a_1^0a_2^0}.
$$
More generally we have
\begin{align}
 G_k(q)={1\over {\displaystyle \prod_{S\con [1,k]}}\Big(1-q
\prod _{i\in S }a_i/{\prod _{j\not\in S
}}a_j\Big)}\Bigg|_{a_1^0a_2^0\cdots a_k^0}, \label{I.12}
\end{align}
where we use (and will often use) $[m,n]$ to denote the set $\{\,
m,m+1,\dots,n\,\}$. Now, standard methods of Invariant Theory yield
that we also have
\begin{align}
 W_k(q )={{ \prod_{i=1}^k}\big(1-{ a_i^2}\big)\over
{\displaystyle \prod_{S\con [1,k]}}\Big(1-q \prod _{i\in S }
a_i/{\prod _{j\not\in S }}a_j\Big)}\bigg |_{a_1^0a_2^0\cdots a_k^0}.
\label{I.13}
\end{align}
A comparison of \eqref{I.12} and \eqref{I.13} strongly suggests that
a close study of the combinatorics of Diophantine systems such as
$\CS_k$ should yield a more revealing path to the construction of
such Hilbert series. This idea turned out to be fruitful, as we
shall see,  in that it permitted the solution  of  a variety of
similar problems (see \cite{2}, \cite{3}). In particular, we were
eventually able to obtain that
\begin{align}
 G_5(\sqrt{q})= {N_5\over (1 -  q)^9(1 -  q^2)^8(1 -
q^3)^6(1 -   q^4)^3(1 -   q^5)}, \label{I.14}
\end{align}
with \begin{small} \begin{align*}
 N_5 &= q^{44} + 7 q^{43} + 220 q^{42} + 2606 q^{41} +
24229 q^{40} + 169840 q^{39} + 951944 q^{38} \cr &\hskip .2in +
4391259 q^{37} + 17128360 q^{36}
 + 57582491 q^{35} + 169556652 q^{34} + 442817680 q^{33}
\cr &\hskip .35in + 1036416952 q^{32} + 2192191607 q^{31} +
4219669696 q^{30} + 7433573145 q^{29} + 12041305271 q^{28} \cr
&\hskip .45in + 18003453305 q^{27} + 24921751416 q^{26} +
32017113319 q^{25} + 38243274851 q^{24} + 42524815013 q^{23} \cr
&\hskip .45in + 44052440432 q^{22} + 42524815013 q^{21} +
38243274851 q^{20} + 32017113319 q^{19} + 24921751416 q^{18} \cr
&\hskip .35in + 18003453305 q^{17} + 12041305271 q^{16} + 7433573145
q^{15} + 4219669696 q^{14} + 2192191607 q^{13} \cr &\hskip .2in +
1036416952 q^{12} + 442817680 q^{11} + 169556652 q^{10} + 57582491
q^{9} + 17128360 q^{8} + 4391259 q^{7} \cr & + 951944 q^{6} + 169840
q^{5} + 24229 q^{4} + 2606 q^{3} + 220 q^{2} + 7 q + 1.
\end{align*}
\end{small}
Surprisingly, the presence of the numerator factor in \eqref{I.13}
absent in \eqref{I.12} does not increase the complexity of the
result, as we see by comparing \eqref{I.14} with the Luque-Thibon
result
\begin{align*}
 W_5(\sqrt{q})= {P_5\over  (1-q^ 2)^4 (1-q^3)  (1-q^4)^6
(1-q^ 5) (1-q^ 6)^ 5 }, 
\end{align*}
with
\begin{align*}
 P_5&=   q^{54}+q^{52}+16 q^{50}+9 q^{49}+98 q^{48}+154
q^{47}+465 q^{46}+915 q^{45}+2042 q^{44}+3794 q^{43}+7263 q^{42} \cr
& \ess\ess+12688 q^{41}+21198 q^{40}+34323 q^{39}+52205 q^{38}
+77068 q^{37}+108458 q^{36}+147423 q^{35}+191794 q^{34} \cr &
  \ess\ess\ess\ess\ess +241863 q^{33}+292689 q^{32}+342207 q^{31}+386980 q^{30}+421057 q^{29}+443990 q^{28}+451398 q^{27}
\cr
&
 \ess\ess    \ess\ess\ess\ess
+443990 q^{26}+421057 q^{25}+386980 q^{24}+342207 q^{23} +292689
q^{22}+241863 q^{21}+191794 q^{20} \cr &
 \ess\ess \ess\ess  \ess\ess\ess\ess +147423 q^{19}+108458 q^{18}
+77068 q^{17}+52205 q^{16}+34323 q^{15}+21198 q^{14}+12688 q^{13}
\cr
&
 \ess\ess\ess\ess \ess\ess\ess\ess \ess\ess
 +7263 q^{12} +3794 q^{11}+2042 q^{10}+915 q^{9}
 +465 q^{8} +154 q^{7}+98 q^{6}+9 q^{5}+16 q^{4}+q^{2}+1
.
\end{align*}
It should be apparent from the size of the numerators of $W_5(q)$
and $G_5(q)$  that the problem of computing these rational functions
explodes beyond $k=4$. In fact it develops that all available
computer packages (including {\ita Omega}  and {\ita Latte}) fail to
directly compute the constant terms in \eqref{I.12}  for $k=5$. This
notwithstanding, we were eventually able to get the partial fraction
algorithm of Xin \cite{9} to deliver us $G_5(q)$.

This paper covers the variety of techniques we developed in  our
efforts to compute these remarkable rational functions. Our efforts
in obtaining $W_6(q)$ and $G_6(q)$ are still in progress, so far
they only resulted in reducing
 the computer time required to obtain $W_5(q)$ and $G_5(q)$.
Using  combinatorial ideas, group actions, in conjunction with the
partial fraction algorithm of  Xin, we developed three essentially
distinct algorithms for computing these rational functions as well
as  other closely related families. Our most successful algorithm
reduces the computation time for $W_5(q)$ down to about five
minutes. The crucial feature of this algorithm  is an inductive
process for successively computing the series $G_k(q)$ and $W_k(q)$,
based on  a surprising role of divided differences.

This paper is the extended version of \cite{new1}. We organize the
contents in 5 sections. Section 1 is this introduction. In Section
\ref{sec-1} we relate these Hilbert series to constant terms and
derive a collection of identities to be used in later sections. In
Section \ref{sec-2} we develop the combinatorial model that reduces
the computation of our Kronecker products to solutions of
Diophantine systems. In Section \ref{sec-3} we develop the divided
difference algorithm for the computation of the complete generating
functions yielding $W_k(q)$ and $G_k(q)$. In Section \ref{sec-4},
after an illustration of what can be done with bare hands we expand
the combinatorial ideas acquired from this experimentation into our
three algorithms that yielded  $G_5(q)$ and our fastest computation
of $W_5(q)$.

The readers are referred to the papers of Luque-Thibon
\cite{4},\cite{5} and Wallach \cite{7} for an understanding of how
these Hilbert series are related to problem arising in the study of
quantum computing.

\section{\label{sec-1}Hilbert series of invariants as constant terms}

 Let us recall that given two
matrices $A=[a_{ij}]_{i,j=1}^m$ and $B=[b_{ij}]_{i,j=1}^n$ we
 use the notation $A\otimes B$ to denote
the $nm\times nm$ block matrix $ A\otimes B= [a_{ij}B]_{i,j=1}^m\,
.$ For instance, if $m=n=2$,
then
$$
A\otimes  B=\left[\mymatrix{ a_{11} b_{11} &  a_{11} b_{12} & a_{12}
b_{11} &  a_{12} b_{12}  \cr a_{11} b_{21} &  a_{11} b_{22} & a_{12}
b_{21} &  a_{12} b_{22}  \cr a_{21} b_{11} &  a_{21} b_{12} & a_{22}
b_{11} &  a_{22} b_{12}  \cr a_{21} b_{21} &  a_{21} b_{22} & a_{22}
b_{21} &  a_{22} b_{22}  \cr }\right].
$$
Here and in the following, we define $T_A P(x)$ to be the action of
an $m\times m$ matrix $A=[a_{ij}]_{i,j=1}^m$ on a polynomial
$P(x)=P(x_1,x_2,\ldots ,x_m)$ in $ \BR_m:=\C[x_1,x_2,\ldots ,x_m] $
by
\begin{align}
T_A P(x_1,x_2,\ldots ,x_m)= P\Big(\sum_{i=1}^m
x_ia_{i1}\scs\sum_{i=1}^m x_ia_{i2}\scs\ldots \scs \sum_{i=1}^m
x_ia_{im}\Big). \label{1.4}
\end{align}
In matrix notation (viewing $x=(x_1,x_2,\ldots ,x_m)$ as a row
vector) we may simply rewrite this as
\begin{align*}
T_A P(x)= P\big(xA\big). 
\end{align*}
Recall that if $G$ is a group of  $m\times m$ matrices we say that
$P$ is \emph{ $G$-invariant} if and only if
\begin{align*}
 T_A P(x)= P(x)\ess\ess\ess\ess \forall\ess\ess\ess A\in G.
\end{align*}
The subspace of $ \BR_m$ of $G$-invariant polynomials is usually
denoted $ \BR_m^G$. Clearly, the action in \eqref{1.4} preserves
homogeneity and degree. Thus we have the direct sum decomposition
$$
 \BR_m^G= \CH_o\big(\BR_m^G\big)\oplus \CH_1\big(\BR_m^G\big)\oplus \CH_2\big(\BR_m^G\big)\oplus\cdots \oplus\CH_d\big(\BR_m^G\big)\oplus\cdots
$$
where $\CH_d\big(\BR_m^G\big)$ denotes the subspace of
$G$-invariants that are homogeneous of degree $d$. The \emph{
Hilbert series} of $\BR_m^G $ is simply given by the formal power
series
\begin{align*}
 F_G(q)= \sum_{d\ge 0} q^d \dim\Big(\CH_d\big(\BR_m^G\big)\Big).
\end{align*}
This is a well defined formal power series since $\dim
\CH_d\big(\BR_m^G\big)\le \dim\Big(\CH_d\big(\BR_m\big)\Big)=
{d+m-1\choose m-1}$.

When $G$ is a finite group the Hilbert series $ F_G(q)$ is
immediately obtained from Molien's formula
$$
F_G(q)= {1\over |G|}\sum_{A\in G}{1\over \det\big( I- qA\big)}.
$$
For an infinite group $G$ which possess a  unit invariant measure
$\om$ this identity becomes
\begin{align}
F_G(q)=  \int_{A\in G}{1\over \det\big( I- qA\big)}\, d\om.
\label{1.8}
\end{align}
For the present developments  we need to specialize all this to the
case $G=SL[2]^{\otimes k}$, that is  the group of $2^k\times 2^k$
matrices obtained by tensoring a $k$-{tuple} of elements of $
SL[2]$. More precisely
\begin{align}
 SL[2]^{\ot k}= \big\{ A_1\otimes A_2\otimes \cdots \otimes A_k  \, :\,   A_i\in SL[2]\ess\ess
\forall \ess\ess i=1,2,\ldots ,k \big\}. \label{1.9}
\end{align}

Our first task in this section is to derive the identity in
\eqref{I.13}. That is

\begin{Theorem}
\label{t-1.1} Setting for $k\ge 1$
\begin{align}
 W_k(q)=  F_{ SL[2]^{\ot k}}(q)= \sum_{d\ge 0} q^d
\dim\Big(\CH_d\big(\BR_{2^k}^{SL[2]^{\ot k}}\big)\Big), \label{1.10}
\end{align}
we have
\begin{align}
 W_k(q)={{ \prod_{i=1}^k}\big(1-{ a_i^2}\big)\over
{\displaystyle \prod_{S\con [1,k]}}\Big(1-q \prod _{i\in S }
a_i/{\prod _{j\not\in S }}a_j\Big)}\bigg |_{a_1^0a_2^0\cdots a_k^0}.
\label{1.11}
\end{align}
\end{Theorem}

We need the following result.
\begin{Proposition} \label{p-1.1} If   $Q(a_1,a_2,\ldots ,a_k)$ is
a Laurent polynomial in $\C[a_1,a_2,\ldots ,a_k; 1/a_1,1/a_2,\ldots
,1/a_k]$ then
\begin{align}
 \left({1\over  2\pi}\right)^k\int_{-\pi}^\pi\cdots
\int_{-\pi}^\pi Q\left(e^{i\ttt_1}, e^{i\ttt_2},\ldots
,e^{i\ttt_k}\right)d\ttt_1d\ttt_2\cdots d\ttt_k = Q(a_1,a_2,\ldots
,a_k)\, \Big|_{a_1^0a_2^0\cdots a_k^0}. \label{1.16}
\end{align}
\end{Proposition}
\begin{proof}

 By multilinearity, it suffices to consider
$Q(a_1,a_2,\ldots,a_k)=a_1^{r_1}a_2^{r_2}\cdots a_k^{r_k}$, in which
case \eqref{1.16} obviously holds.
\end{proof}

\begin{proof}[Proof of Theorem \ref{t-1.1}]

To keep our exposition within reasonable limits we will need  to
assume here some well known facts (see \cite{7} for proofs). Since $
SL[2]$ has no finite measure the first step is to note that a
polynomial \hbox{$P(x)\in \C[x_1,x_2,\ldots ,x_{2^k}]$}
 is $ SL[2]^{\ot k}$-invariant if and only if it is  $SU[2]^{\ot k}$-invariant, where $SU[2]:=SU(2,\C)$ and as in
 \eqref{1.9}
$$
 SU[2]^{\ot k}= \big\{ A_1\ot A_2\ot \cdots \ot A_k  \, :\, A_i\in SU[2]\ess\ess  \forall \ess\ess i=1,2,\ldots ,k
 \big\}.
$$
In particular we derive that $
 F_{ SL[2]^{\ot k}}(q)=  F_{ SU[2]^{\ot k}}(q).
$
This fact allows us to compute $F_{ SL[2]^{\ot k}}(q)$ using
Molien's identity \eqref{1.8}. Note however that if
$$
A=  A_1\ot A_2\ot \cdots \ot A_k
$$
and $A_i$ has eigenvalues $\ssp t_i,1/t_i\ssp $ then (using plethistic notation) we have
\begin{align*}
 {1\over \det\big(I-qA)}= \sum_{m\ge 0}q^m
h_m\big[(t_1+1/t_1)(t_2+1/t_2)\cdots (t_k+1/t_k)\big]. 
\end{align*}
Denoting by $d\om_i$ the invariant measure of the $i^{\mathrm{th} }$
copy of $SU[2]$ we see that \eqref{1.8} reduces to
\begin{align}
 F_{ SU[2]^{\ot k}}(q)= \sum_{m\ge 0}q^m\int_{  SU[2]}\cdots \int_{  SU[2]}
 h_m\big[(t_1+1/t_1) \cdots (t_k+1/t_k)\big] d\om_1  \cdots d\om_k.
\label{1.14}
\end{align}
Now it is well know that if an integrand $f(A)$ of $SU[2]$ is
invariant under conjugation then
$$
\int_{  SU[2]}f(A)d\om= {1\over \pi}\int_{-\pi}^\pi f\Big(
\left[\mymatrix{e^{i\ttt}& 0 \cr 0 &e^{-i\ttt}}\right] \Big) \sin^2
\ttt d\ttt.
$$
This identity converts the right-hand side of \eqref{1.14} to
\begin{align}
\sum_{m\ge 0}q^m {1\over \pi^k} \int_{-\pi}^\pi\!\! \cdots
\int_{-\pi}^\pi
 h_m\big[(e^{i\ttt_1}+e^{-i\ttt_1}) \cdots (e^{i\ttt_k}+e^{-i\ttt_k})\big] \sin^2 \ttt_1 \cdots \sin^2 \ttt_k\,\,  d\ttt_1  \cdots
 d\ttt_k.
\label{1.15}
\end{align}


%
%
%
\sas

The substitution
$$
\sin^2\ttt_j={1-{e^{2i\ttt_j}+e^{-2i\ttt_j}\over 2}\over 2}
$$
reduces the coefficient of $q^m$ to
\begin{align}
 \ess {1\over (2\pi)^k} \int_{-\pi}^\pi \cdots \int_{-\pi}^\pi
 h_m\big[(e^{i\ttt_1}+e^{-i\ttt_1}) \cdots (e^{ i\ttt_k}+e^{ - i\ttt_k})\big]
\prod_{i=1}^k\Big( 1-{e^{2i\ttt_j}+e^{-2i\ttt_j}\over 2} \Big)\,\,
 d\ttt_1  \cdots d\ttt_k.
\label{1.17}
\end{align}
However the factor $
 h_m\big[(e^{i\ttt_1}+e^{-i\ttt_1}) \cdots (e^{ i\ttt_k}+e^{ - i\ttt_k})\big]
$
is invariant under any of the interchanges
$e^{i\ttt_j}\longleftrightarrow e^{-i\ttt_j}$. Thus the integral in
\eqref{1.17} may be simplified to
$$
\ess {1\over (2\pi)^k}
\int_{-\pi}^\pi \cdots \int_{-\pi}^\pi
 h_m\big[(e^{i\ttt_1}+e^{-i\ttt_1}) \cdots (e^{ i\ttt_k}+e^{ - i\ttt_k})\big]
\prod_{i=1}^k\Big( 1- e^{2i\ttt_j} \Big)\,\,
 d\ttt_1  \cdots d\ttt_k.
$$
Proposition \ref{p-1.1} then yields that this integral may be
computed as the constant term
$$
h_m\big[(a_1+1/a_1)(a_2+1/a_2)\cdots (a_k+1/a_k)\big
]\prod_{i=1}^k\Big( 1- a_i^2  \Big)\,\, \bigg|_{a_1^0a_2^0\cdots
a_k^0}.
$$
Using this in \eqref{1.15} we derive that
\begin{align*}
 F_{ SU[2]^{\ot k}}(q) &=  \sum_{m\ge 0}q^m
h_m\big[(a_1+1/a_1)(a_2+1/a_2)\cdots (a_k+1/a_k)\big
]\prod_{i=1}^k\Big( 1- a_i^2  \Big)\,\,\bigg|_{a_1^0a_2^0\cdots
a_k^0} \cr & = \sum_{m\ge 0}q^m
h_m\bigg[\sum_{S\con[1,k]}{\prod_{i\in S} a_i\over \prod_{j\not\in
S} a_j  }\bigg ]\prod_{i=1}^k\Big( 1- a_i^2
\Big)\,\,\bigg|_{a_1^0a_2^0\cdots a_k^0} \cr &=
\Bigg(\prod_{S\con[1,k]}{1\over \Big( 1-q\,{ \prod_{i\in S}
a_i\over\prod_{j\not\in S} a_j } \Big)} \Bigg) \prod_{i=1}^k\Big( 1-
a_i^2  \Big)\,\,\bigg|_{a_1^0a_2^0\cdots a_k^0}.
\end{align*}
This completes the proof of Theorem \ref{t-1.1}.
\end{proof}

Note that if we restrict our action of $SU[2]^{\otimes k}$ to the subgroup of matrices
$$
T_2^{\otimes k}= \left\{\left[\mymatrix{t_1  & 0\cr 0 & {\overline
t_1}\cr}\right]\otimes \left[\mymatrix{t_2  & 0\cr 0 & {\overline
t_2}\cr}\right]\otimes\cdots \otimes \left[\mymatrix{t_k & 0\cr 0 &
{\overline t_k}\cr}\right]\ssp :\ssp t_r=e^{i\ttt_r} \right\}
$$
then a similar use of  Molien's theorem yields the following
result.\sas

\begin{Theorem}
\label{1.2}

The Hilbert series of the ring of invariants
$\BR_{2^k}^{T_2^{\otimes k}}$ is given by the constant term
\begin{align}
 F_{T_2^{\otimes k}}(q)= { 1\over {\displaystyle \prod_{S\con
[1,k]}}\Big(1-q \prod _{i\in S } a_i/{\prod _{j\not\in S
}}a_j\Big)}\bigg |_{a_1^0a_2^0\cdots a_k^0}. \label{1.19}
\end{align}
\end{Theorem}
\begin{proof}
 The integrand $1/\det(1-qA)$ is the same as in the previous proof
and only  the Haar measure changes. In this case we must take $dw=
d\ttt_1 d\ttt_2\cdots d\ttt_k/(2\pi)^k$ in \eqref{1.8}, and Molien's
theorem  gives
$$
F_{T_2^{\otimes k}}(q)={1\over (2\pi)^k}\int_{-\pi}^{\pi} \cdots
\int_{-\pi}^{\pi} { 1\over {\displaystyle \prod_{S\con
[1,k]}}\Big(1-q \prod _{i\in S } t_i/{\prod _{j\not\in S }}t_j\Big)}
d\ttt_1 d\ttt_2\cdots d\ttt_k.
$$
Thus \eqref{1.19} follows from Proposition \ref{p-1.1}.
\end{proof}

\begin{Remark}\label{rem-1.2}
There is another path leading to the same result that is worth
mentioning here since it gives a direct way of connecting Invariants
to Diophantine systems. For notational simplicity we will deal with
the case $k=3$. Note that the element
$$
 \left[
\mymatrix{ t_1  & 0\cr
 0 & {\overline t_1}\cr
}\right]\otimes
 \left[
\mymatrix{ t_2  & 0\cr
 0 & {\overline t_2}\cr
}\right]\otimes
 \left[
\mymatrix{ t_3  & 0\cr
 0 & {\overline t_3}\cr
}\right] \in  {T_2^{\otimes 3}}
$$
is none other than the $8\times 8$ diagonal matrix
$$
A(t_1,t_2,t_3)= \left[\mymatrix{ t_1t_2t_3 &  0 &  0 &  0 &  0 &
0 &  0 &  0 &  \cr 0 &  t_1t_2/t_3 &  0 &  0 &  0 &  0 &  0 &  0 &
\cr 0 &  0 &  t_1t_3/t_2 &  0 &  0 &  0 &  0 &  0 &  \cr 0 &  0 &  0
& t_1/t_2t_3 &  0 &  0 &  0 &  0 &  \cr 0 &  0 &  0 &  0 &
t_2t_3/t_1 &  0 &  0 &  0 &  \cr 0 &  0 &  0 &  0 & 0 & t_2/t_1t_3 &
0 &  0 &  \cr 0 &  0 &  0 &  0 &  0 &  0 &t_3/ t_1t_2 &  0 &  \cr 0
&  0 &  0 &  0 &  0 &  0 &  0 &  1/t_1t_2t_3 &  \cr }\right].
$$
This gives that for any monomial $x^p=x_1^{p_1}x_2^{p_2}\cdots x_8^{p_8}$ we have
$$
A(t_1,t_2,t_3)\, x^p= t_1^{p_1+p_2+p_3+p_4-p_5-p_6-p_7-p_8} \,
t_2^{p_1+p_2-p_3-p_4+p_5+p_6-p_7-p_8} \,
t_3^{p_1-p_2+p_3-p_4+p_5-p_6+p_7-p_8} \times x^p.
$$
Thus all the monomials are eigenvectors and a polynomial
$P(x_1,x_2,\ldots ,x_8)$ will be invariant if and only if  all its
monomials are eigenvectors of eigenvalue 1. It then follows that the
Hilbert series $F_{T_2^{\otimes 3}}(q)$ of $\C[x_1,x_2,\ldots
,x_8]^{T_2^{\otimes 3}}$ is obtained by $q$-counting these monomials
by total degree. That is $q$-counting by the statistic
$p_1+p_2+p_3+p_4+p_5+p_6+p_7+p_8$ the solutions of the Diophantine
system
\begin{align}
 \CS_3= \left\|\  \mymatrix{ p_1+p_2+p_3+p_4-p_5-p_6-p_7-p_8=0\cr
p_1+p_2-p_3-p_4+p_5+p_6-p_7-p_8=0\cr
p_1-p_2+p_3-p_4+p_5-p_6+p_7-p_8=0\cr }\right. \label{1.20}
\end{align}
and MacMahon partition analysis gives
$$
\textstyle F_{T_2^{\otimes 3}}(q)= {1\over 1- q a_1a_2a_3} {1\over
1- q a_1a_2/a_3} {1\over 1- q a_1a_3/a_2} {1\over 1- q a_1/a_2a_3}
{1\over 1- q a_2a_3/a_1} {1\over 1- q a_2/a_1a_3} {1\over 1- q a_3/
a_1a_2} {1\over 1- q a/a_1a_2a_3} \Big|_{a_1^0a_2^0a_3^0}.
$$
This gives another proof of the case $k=3$ of \eqref{1.19}. It is
also clear that the same argument can be used for all $k>3$ as well.
\end{Remark}

\begin{Remark}
\label{rem-1.3} Full information about the solutions of our systems
is given  by the complete generating function
\begin{align}
 F_k(x_1,x_2,\ldots ,x_{2^k})= \sum_{p\in\CS_k}
x_1^{p_1}x_2^{p_2}\cdots x_{2^k}^{p_{2^k}}. \label{1.21}
\end{align}
Using the notation adopted for $\CS_3$ in \eqref{1.20}, our system
$\CS_k$ may be written in vector form
$$
p_1V_1\sps  p_2V_2 \sps \cdots \sps  p_{2^k}V_{2^k}= 0,
$$
where $V_1,V_2,\ldots ,V_{2^k}$ are the $k$-vectors $(\pm 1,\pm 1,
\ldots, \pm 1) $ yielding the vertices of the hypercube of semiside
$1$ centered at the origin. In this notation,
 MacMahon partition analysis gives that the rational function in \eqref{1.21} is obtained by
taking the constant term
$$
F_k(x_1,x_2,\ldots ,x_{2^k})= \prod_{i=1}^{2^k}{1\over
1-x_iA_i}\bigg|_{a_1^0 a_2^0\cdots a_k^0}
$$
with the $A_i$ Laurent monomials in $a_1,a_2,\ldots a_k$ which may be written in the form
$$
A_i = \prod_{i=1}^k a_i^{1-2\eee_i} \eqno
$$
where $\eee_1\eee_2\cdots \eee_k$ are the binary digits of $i-1$.

In the same vein the companion rational  function $W(x_1,x_2,\ldots ,x_{2^k})$ associated to
the  Sdd problem is obtained by taking  the constant term
\begin{align}
\label{1.18}
 W_k(x_1,x_2,\ldots ,x_{2^k})= \prod_{j=1}^k(1-a_j^2)
\prod_{i=1}^{2^k}{1\over 1-x_iA_i}\bigg|_{a_1^0 a_2^0\cdots a_k^0}.
\end{align}
Of course we have
$$
G_k(q)= F_k(x_1,x_2,\ldots ,x_{2^k})\Big|_{x_i=q}  \ess\ess
\hbox{and} \ess\ess W_k(q)= W_k(x_1,x_2,\ldots
,x_{2^k})\Big|_{x_i=q}.
$$
In Section \ref{sec-3} we will show that, at least in principle,
these rational functions could be constructed by a succession of
elementary steps interspersed by single constant term extractions.
\end{Remark}

\section{\label{sec-2}Diophantine systems,  Constant terms and Kronecker
products}
 We have seen, by MacMahon
partition analysis, that the generating function $G_k(q)$ defined in
\eqref{I.10},
which counts solutions of the Diophantine system $\CS_k$, is given
by the constant term identity in \eqref{I.12}:
\begin{align}
 G_k(q)={1\over {\displaystyle \prod_{S\con [1,k]}} \Big(1-q
\prod _{i\in S }a_i/{\prod _{j\not\in S
}}a_j\Big)}\bigg|_{a_1^0a_2^0\cdots a_k^0}. \label{2.1}
\end{align}
In the last section we proved (in Theorem \ref{t-1.1}) that the
Hilbert series $W_k(q)$ of invariants in \eqref{1.10}
is given by the constant term
\begin{align}
W_k(q)={{ \prod_{i=1}^k}\big(1-{ a_i^2}\big)\over {\displaystyle
\prod_{S\con [1,k]}}\Big(1-q \prod _{i\in S } a_i/{\prod _{j\not\in
S }}a_j\Big)}\bigg |_{a_1^0a_2^0\cdots a_k^0}.\label{2.2}
\end{align}
A comparison of \eqref{2.1} and \eqref{2.2} clearly suggests that
these two results must be connected. This connection has a beautiful
combinatorial underpinning which leads to another interpretation of
the these remarkable constant terms.
  The idea is best explained  in the simplest case $k=2$. Then \eqref{2.2} reduces to
$$
W_2(q)= {1-a_1^2-a_2^2+a_1^2a_2^2\over (1-q a_1 a_2)(1-q
a_1/a_2)(1-q a_2/a_1)(1-q /a_1a_2)}\bigg |_{a_1^0a_2^0}.
$$
Expanding the inner rational function as product of four formal
power series in $q$ we get
\begin{align}
W_2(q)&= \sum_{p_{00}\ge 0}\sum_{p_{01}\ge 0}\sum_{p_{10}\ge
0}\sum_{p_{11}\ge 0}q^{p_{00}+p_{01}+p_{10}+p_{11}}
a_1^{p_{00}+p_{01}-p_{10}-p_{11}}a_2^{p_{00}-p_{01}+p_{10}-p_{11}}\ess
\bigg |_{a_1^0a_2^0} \cr &\qquad- \sum_{p_{00}\ge 0}\sum_{p_{01}\ge
0}\sum_{p_{10}\ge 0}\sum_{p_{11}\ge
0}q^{p_{00}+p_{01}+p_{10}+p_{11}}
a_1^{p_{00}+p_{01}-p_{10}-p_{11}+2}a_2^{p_{00}-p_{01}+p_{10}-p_{11}}\ess
\bigg |_{a_1^0a_2^0} \label{2.3}\\ &\qquad\quad - \sum_{p_{00}\ge
0}\sum_{p_{01}\ge 0}\sum_{p_{10}\ge 0}\sum_{p_{11}\ge
0}q^{p_{00}+p_{01}+p_{10}+p_{11}}
a_1^{p_{00}+p_{01}-p_{10}-p_{11}}a_2^{p_{00}-p_{01}+p_{10}-p_{11}+2}\ess
\bigg |_{a_1^0a_2^0} \cr &\qquad\qquad+ \sum_{p_{00}\ge
0}\sum_{p_{01}\ge 0}\sum_{p_{10}\ge 0}\sum_{p_{11}\ge
0}q^{p_{00}+p_{01}+p_{10+p_{11}}}
a_1^{p_{00}+p_{01}-p_{10}-p_{11}+2}a_2^{p_{00}-p_{01}+p_{10}-p_{11}+2}\ess
\bigg |_{a_1^0a_2^0}\nonumber
 .
\end{align}
Now by MacMahon partition analysis, the the $i^{\textrm{th}}$ term
counts solutions of the Diophantine system
\begin{align}
 \CS_2^i= \left\|\ {p_{00}+p_{01}-p_{10}-p_{11}= c_i \atop
p_{00}-p_{01}+p_{10}-p_{11}= d_i }\right.\ \  , \label{2.4}
\end{align}
where $(c_i,d_i)$ equals $(0,0),(-2,0),(0,-2),(-2,-2)$ for
$i=1,2,3,4$, respectively. Note that the first term of \eqref{2.3}
is none other than \eqref{2.1} for $k=2$.

Applying the same decomposition in the general case we see that the
series $W_k(q)$ may be viewed as the end product of an inclusion
exclusion process applied to a family of Diophantine systems. To
derive some further consequences of this fact, it is more convenient
to use another combinatorial model for these systems. In this
alternate model our family  of objects  consists of the collection
$\CF_d$ of $d$-subsets of the $2d$-element set
$$
\OM_{2d}=\{1,2,3,\ldots ,2d\}.
$$
For a given $ A=\{1\le i_1<i_2<\cdots <i_d\le 2d\}\in \CF_d$  and $\sig$ in the symmetric group $ S_{2d}$ we set
$$
\sig   A= \{\sig_{i_1},\sig_{i_2},\ldots ,\sig_{i_d}\}.
$$
This clearly defines an action of $ S_{2d}$ on $\CF_d$ as well as on the $k$-fold cartesian product
$$
 \CF_d^{k}= {\CF_d \times \CF_d \times \CF_d \times \cdots \times \CF_d
 }.
$$

\begin{Theorem}
\label{t-2.1} The number $m_d(k)$ of solutions of the Diophantine
system $\CS_k$ is equal to the number of orbits in the action of
$S_{2d}$ on $\CF_d^{ k}$.
\end{Theorem}
\begin{proof}

 It will be  sufficient to see this for $k=2$.  Then leaving $d$ generic  we can
 visualize an element of $\Fd \times \Fd$ by the  Ven diagram of Figure
 \ref{fig-Fd2}.
 There we have depicted  the pair $(A_1,A_2)$ as it lies in $\OM_{2d}$.
 Using these two sets we can decompose  $\OM_{2d}$ into $4$ parts
 labeled by
 $ A_{00},A_{01}, A_{10}, A_{11}
 $.
More precisely ``$ A_{ 00}$'' labels the set $A_1\cap A_2$, ``$
A_{01}$'' labels the set $A_1\cap\; ^c\!A_2$, ``$  A_{10}$'' labels
the set $^c\! A_1\cap A_2$ and  ``$  A_{11}$'' labels the set
$^c\!A_1\cap\; ^c\!A_2$. Here  we use ``$^c\!A_i\, $'' to denote the
complement of  $ A_i$ in $ \OM_{2d}$.
 This given, if we let $p_{00},p_{01},p_{10},p_{11}$
denote the respective cardinalities   of these sets,  the
 condition that the pair $(A_1,A_2)$  belongs to $\Fd\times \Fd$ yields that
we must have
$$
\eqalign { &p_{00}+p_{01}+p_{10}+p_{11} = 2d\cr &p_{00}+p_{01} =
|A_1|=   d\cr & p_{00}+p_{10} = |A_2|=  d\cr }\ \ \ .
$$
Note that this system of equations is equivalent to the system
$$
\eqalign { &p_{00}+p_{01}+p_{10}+p_{11} = 2d\cr &p_{00}+p_{01} -
p_{10}-p_{11}=   0\cr & p_{00} -  p_{01}+p_{10 }-p_{11 }  = 0\cr }\
\ \ .
$$

\begin{figure}[hbt]
\includegraphics[width=5cm]{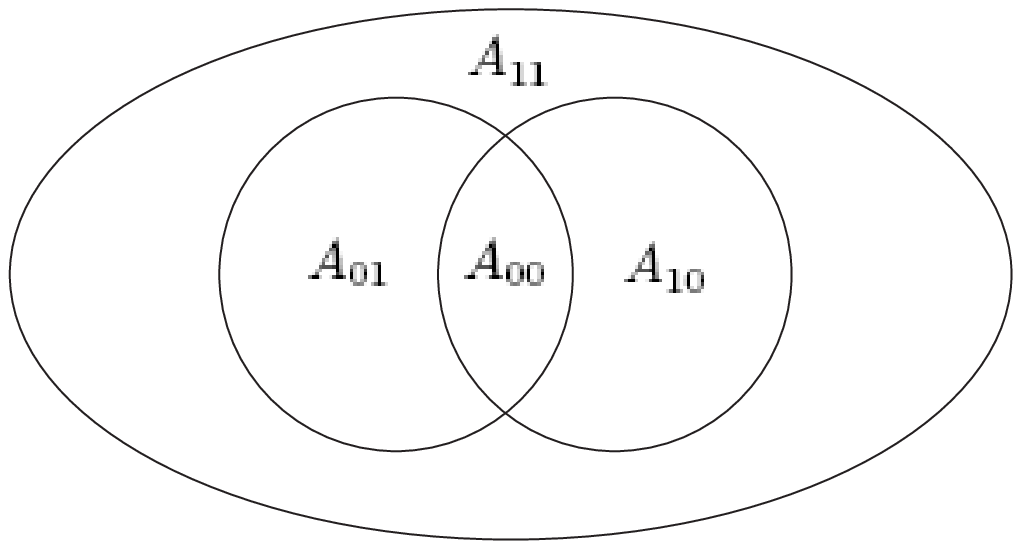}
\caption{The Ven diagram for $\Fd^{ 2}$.\label{fig-Fd2}}
\end{figure}

It is easily seen that for any solution
$(p_{00},p_{01},p_{10},p_{11} )$ of this system,  we can immediately
construct a pair of subsets $(A_1,A_2)\in\Fd\times \Fd$ by simply
filling the sets $ A_{00},A_{01},A_{10},A_{11}$ in the diagram of
Figure \ref{fig-Fd2}  with $p_{00},p_{01},p_{10},p_{11}$ respective
elements from the set $\OM_{2d}$. Moreover, any two such fillings
can be seen to be images of each other under suitable permutations
of $S_{2d}$. In other words by this construction we obtain a
bijection between the orbits of $\Fd\times \Fd$ under $\CS_{2d}$ and
the  solutions of the system $\CS_2$ we have previously encountered.
This proves the theorem for $k=2$. The general case follows by an
entirely analogous argument.
\end{proof}

Now we are ready to prove Theorem \ref{t-I.2} and then Theorem
\ref{t-I.1}.
\begin{proof}[Proof of Theorem \ref{t-I.2}]
We are to show that
\begin{align}
m_k(d)= \LL h_{d,d} \odot h_{d,d} \odot\cdots \odot h_{d,d} \scs
s_{2d}\RR. \label{2.8}
\end{align}
It is well known that
 a transitive  action of a group $G$ on a set $\OM$ is equivalent to the action of $G$ on the left $G$-cosets
of the stabilizer of any element of $\OM$. In our case, pick the
subset $[1,d]$ of $\OM_{2d}$. Then the stabilizer is the Young
subgroup $S_{[1,d]}\times S_{[d+1,2d]}$ of $S_{2d}$ and thus the
Frobenius characteristic of this action is the homogeneous basis
element $h_{d,d}=h_d h_d$.  It follows then that the Frobenius
characteristic of the action of $S_{2d}$ on the $k$-tuples
$(A_1,A_2,\ldots ,A_k)$ of  $d$-subsets of $\OM_{2d}$ is given by
the $k$-fold Kronecker product $
 h_{d,d} \odot h_{d,d} \odot\cdots \odot h_{d,d}.
$ Therefore the scalar product
$$
\LL h_{d,d} \odot h_{d,d} \odot\cdots \odot h_{d,d} \scs s_{2d}\RR
$$
yields the multiplicity of the trivial under this action. But it is
well known, and easy to see that this multiplicity is also equal to
the number of orbits under this action. Thus \eqref{2.8} follows by
Theorem \ref{t-2.1}. \end{proof}

\begin{proof}[Proof of Theorem \ref{t-I.1}]
Again we will only need to do it for $k=2$. To this end note that by
Theorem \ref{t-I.2} the number of solutions of the system $\CS_2^1$
in \eqref{2.4} is given by the scalar product
\begin{align}
\LL h_{d,d}\odot h_{d,d}\scs s_{2d}\RR. \label{2.9}
\end{align}
In the same vein we see that the number of solutions to the system
$\CS_2^2$ in \eqref{2.4} may be viewed as the number of orbits in
the action of $S_{2d}$ on the pairs of subsets $(A_1,A_2)$ of
$\OM_{2d}$ where $|A_2|=|^cA_2|$ and $|A_1|=|^cA_1|+2$. We have seen
that the Frobenius characteristic of the action of $S_{2d}$ on
subsets of cardinality $d$ is $h_{d,d}$. On the other hand the
action of $S_{2d}$ on sets of cardinality $d+1$ is equivalent to the
action of $S_{2d}$ on left cosets of $S_{[1,d+1]}\times S_{[d+2,2d]}
$ yielding that the Frobenius characteristic for this action is
$h_{d+1}h_{d-1}$. Thus the Frobenius characteristic of the action of
$S_{2d}$ on such pairs must be the Kronecker product
$$
h_{d+1}h_{d-1}\odot h_{d }h_{d }.
$$
It then follows that the number of solutions of the system $\CS_2^2$
is given by the scalar product
\begin{align}
 \LL h_{d+1}h_{d-1}\odot h_{d }h_{d }\scs s_{2d}\RR. \label{2.10}
\end{align}
The same reasoning gives that the number of solutions of the systems
$\CS_2^3$ and $\CS_2^4$ in \eqref{2.4} are given by the scalar
products
\begin{align}
 \LL  h_{d }h_{d }\odot h_{d+1}h_{d-1} \scs s_{2d}\RR
\ess\ess\hbox{and } \ess\ess \LL h_{d+1}h_{d-1}\odot
h_{d+1}h_{d-1}\scs s_{2d}\RR. \label{2.11}
\end{align}
It follows then that the coefficient of $q^{2d}$ in the alternating
sum of formal power series in \eqref{2.3} is none other than the
following alternating sum of the scalar products in \eqref{2.9},
\eqref{2.10} and \eqref{2.11}.
 \begin{align*}
 W_2(q)\Big|_{q^{2d}} &=   \LL
h_dh_d\odot h_dh_d\scs s_{2d}\RR- \LL h_{d+1}h_{d-1}\odot h_{d }h_{d
}\scs s_{2d}\RR \\&\qquad\qquad\qquad\qquad\qquad\qquad- \LL h_{d
}h_{d }\odot h_{d+1}h_{d-1} \scs s_{2d}\RR + \LL h_{d+1}h_{d-1}\odot
h_{d+1}h_{d-1}\scs s_{2d}\RR \cr &= \LL\big( h_dh_d-
h_{d+1}h_{d-1}\big)\odot \big(h_dh_d- h_{d+1}h_{d-1}\big)\scs
s_{2d}\RR= \LL s_{ d,d }\odot s_{ d,d }\scs s_{2d}\RR.
\end{align*}
Summing over $d$ gives
$$
W_2(q)= \sum_{d\ge 0}q^{2d}\LL s_{ d,d }\odot s_{ d,d }\scs
s_{2d}\RR.
$$
An  entirely analogous argument proves  the general identity in
\eqref{I.8}. \end{proof}

\section{\label{sec-3}Enter divided difference operators}
 There is a truly remarkable
approach to the solutions of a variety of constant term problems
which exhibit the same types of symmetries of the Hdd and Sdd
problems. We will introduce the approach in some simple cases first.
We define the \emph{ double} of the Diophantine system
$$
\CS_2= \left\|\ess \mymatrix{ p_1+p_2-p_3-p_4=0\cr
p_1-p_2+p_3-p_4=0\cr }\right. 
$$
to be the system
$$
\CS\CS_2= \bigg\|\ess \mymatrix{ p_1+p_2-p_3-p_4\sps
p_5+p_6-p_7-p_8=0\cr p_1-p_2+p_3-p_4\sps p_5-p_6+p_7-p_8=0\cr }.
$$
As we can easily see we have simply repeated twice each linear form  and appropriately increased the
indices of the variables. Now suppose that we are in possession of the complete generating function of $\CS_2$,
that is
$$
F_{\CS_2}(x_1,x_2,x_3,x_4)= \sum_{p\in \CS_2
}x_1^{p_1}x_2^{p_2}x_3^{p_3}x_4^{p_4}. 
$$
We claim that the complete generating function of $\CS\CS_2$ is
simply given by
\begin{align}
 F_{\CS\CS_2}(x_1,x_2,\ldots ,x_8)=
\dd_{1,5}\dd_{2,6}\dd_{3,7}\dd_{4,8} F_{\CS_2}(x_1,x_2,x_3,x_4),
\label{3.4}
\end{align}
where for any pair of indices $(i,j)$ we let $\dd_{i,j}$ denote the
divided difference operator defined for any function $f(x)$ by
$$
\dd_{i,j} f(x) = \frac{f(x)-f(x)\big|_{x_i=x_j,x_j=x_i}}{ x_i-x_j}.
$$

\begin{proof}[Proof of \eqref{3.4}]
 By MacMahon partition analysis we have
\begin{align}
 F_{\CS_2}(x_1,x_2,x_3,x_4)= {1\over (1-x_1 a_1a_2)}{1\over
(1-x_2 a_1/a_2)} {1\over (1-x_3 a_2/a_1)}{1\over (1-x_4/
a_1a_2)}\ssp \bigg|_{a_1^0a_2^0}. \label{3.6}
\end{align}
Now note that since \begin{align*}
 \dd_{1,5}{1\over (1-x_1
a_1a_2)} &= \left({1\over (1-x_1 a_1a_2)}\sms {1\over (1-x_5
a_1a_2)}\right){1\over x_1-x_5} 
= {a_1a_2 \over  (1-x_1 a_1a_2) (1-x_5 a_1a_2)},
\end{align*}
we obtain similarly
$$
\dd_{2,6}{1\over (1-x_2 a_1/a_2)}= {a_1/a_2 \over  (1-x_2 a_1/a_2)
(1-x_6 a_1/a_2)},
$$
$$
\dd_{3,7}{1\over (1-x_3 a_2/a_1)}= {a_2/a_1 \over  (1-x_3 a_2/a_1)
(1-x_7 a_2/a_1)},
$$
$$
\dd_{4,8}{1\over (1-x_4 /a_1a_2)}= {1/a_1a_2 \over  (1-x_4 /a_1a_2)
(1-x_8/ a_1a_2)}.
$$
Thus applying the operator $\dd_{1,5}\dd_{2,6}\dd_{3,7}\dd_{4,8}$ to
both sides of  \eqref{3.6}  gives
\begin{multline}
 \dd_{1,5}\dd_{2,6}\dd_{3,7}\dd_{4,8} F_{\CS_2}(x_1,x_2,x_3,x_4)
=
 {1\over (1-x_1 a_1a_2)
(1-x_2 a_1/a_2) (1-x_3 a_2/a_1) (1-x_4  /a_1a_2)} \\
 {\over (1-x_5 a_1a_2)
 (1-x_6 a_1/a_2)
(1-x_7 a_2/a_1) (1-x_8 /a_1 a_2)}\bigg|_{a_1^0a_2^0}.\label{3.7}
\end{multline}
 Now we can easily recognize that \eqref{3.7} is precisely the constant term that
MacMahon partition analysis would yield for the system $\CS\CS_2$.
This proves \eqref{3.4}.
\end{proof}

Note that to obtain the equality in \eqref{3.7} we have used the
simple fact that the divided difference operator and the constant
term operator do commute. This is the fundamental property which is
at the root of the present algorithm. This example should make it
evident to have the following more general result (with \emph{
double} modified).

\begin{Theorem}
\label{t-3.1} If $F_\CS(\xon ) $ is the complete generating function
of the Diophantine system
$$\CS = \left\| \ \begin{bmatrix}
                    b_{11} & b_{12} & \cdots & b_{1n} \\
                    \vdots & \vdots & \cdots & \vdots \\
                    b_{r1} & b_{r2} & \cdots & b_{rn} \\
                  \end{bmatrix}
                  \begin{bmatrix}
                    p_1 \\
                    \vdots \\
                    p_n \\
                  \end{bmatrix} = \begin{bmatrix}
                                    c_1 \\
                                    \vdots \\
                                    c_r \\
                                  \end{bmatrix} \right.
, $$ then the complete generating function of the \emph{doubling} of
$\CS$ defined by
$$\CS\CS = \left\| \ \left[\begin{matrix}
                    b_{11} & b_{12} & \cdots & b_{1n} \\
                    \vdots & \vdots & \cdots & \vdots \\
                    b_{r1} & b_{r2} & \cdots & b_{rn} \\
                  \end{matrix}\ \right|\left.\begin{matrix}
                    b_{11} & b_{12} & \cdots & b_{1n} \\
                    \vdots & \vdots & \cdots & \vdots \\
                    b_{r1} & b_{r2} & \cdots & b_{rn} \\
                  \end{matrix} \right]
                  \begin{bmatrix}
                    p_1 \\
                    \vdots \\
                    p_{2n} \\
                  \end{bmatrix} = \begin{bmatrix}
                                    c_1-b_{11}-b_{12}-\cdots-b_{1n} \\
                                    \vdots \\
                                    c_r-b_{r1}-b_{r2}-\cdots-b_{rn} \\
                                  \end{bmatrix} \right.
 $$
is given by the rational function
$$
F_{\CS\CS}(x_1,x_2,\dots,x_{2n} )= \dd_{1,n+1}\dd_{2,n+2}\cdots
\dd_{n,2n}F_\CS(\xon ).
$$
\end{Theorem}

This result combined with the next simple observation yields a
powerful algorithm for computing a variety of complete generating
functions.

\begin{Theorem}
\label{3.2} Let $F_\CS(\xon ) $ be the complete generating function
of a Diophantine system $\CS$ then the complete generating function
$F_{\CS\CE}(\xon)$ of the system $\CS\CE$ obtained by adding the
equation
$$
\CE  =\big\|\  r_1 p_1+r_2 p_2+\cdots +r_n p_n=s
$$
to $\CS$ is obtained by taking the constant term
$$
F_{\CS\CE}(\xon)= a^{-s}F_\CS(a^{r_1}x_1,a^{r_2}x_2,\ldots ,
a^{r_n}x_n
)\Big|_{a^0}. 
$$
\end{Theorem}
\begin{proof}
 By assumption
$$
F_\CS(\xon)= \sum_{p\in \CS }x_1^{p_1}x_2^{p_2}\cdots x_n^{p_n}.
$$
Now we have \begin{align*}  a^{-s}F_\CS(a^{r_1}x_1,a^{r_2}x_2,\ldots
, a^{r_n}x_n )\Big|_{a^0} &=\sum_{p\in \CS }x_1^{p_1}x_2^{p_2}\cdots
x_n^{p_n} a^{r_1 p_1+r_2 p_2+\cdots +r_n p_n-s}\Big|_{a^0} \cr
&=\sum_{p\in \CS\CE }x_1^{p_1}x_2^{p_2}\cdots x_n^{p_n} \cr &=
F_{\CS\CE}(\xon).\bigsp\bigsp\bigsp 
\end{align*}

\vspace{-20pt}\
\end{proof}
These two results provide us with algorithms for (at least in
principle) computing all the Hdd series $G_k(q)$
as well as the Sdd series $W_k(q)$.
\begin{Algorithm}[Hdd Case]
\label{algo-Fk}

\begin{enumerate}
\myitem {$\bf a_1)$} Initially compute the complete generating
function for the Hdd problem for $k=1$. That is, compute the
constant term
$$
F_1(x_1,x_2)={1\over (1-  x_1a)(1-x_2/a)}\ssp \Big|_{a^0}.
$$

\myitem{$\bf a_k)$} With $F_{k-1}(x_1,\ldots ,x_{2^{k-1}})$ from
step $\bf b_{k-1})$, compute by divided difference
$$
FF_{{k-1}}(x_1,\ldots ,x_{2^{k}})= \dd_{1,1+2^{k-1}}\cdots
\dd_{2^{k-1},2^{k}}F_{k-1}(x_1,\ldots ,x_{2^{k-1}}).
$$

\myitem{$\bf b_k )$} With $FF_{k-1}(x_1,\ldots ,x_{2^{k-1}})$ from
step $\bf a_{k})$,  compute the complete generating function for the
Sdd problem for $k$ by the following constant term:
$$
F_k(x_1,x_2,\ldots,x_{2^k})= FF_{k-1}(ax_1,ax_2,\ldots
,ax_{2^{k-1}},x_{2^{k-1}+1}/a,\ldots ,x_{2^k}/a)\Big |_{a^0}.
$$
\end{enumerate}
\end{Algorithm}

This sequence of steps in Algorithm \ref{algo-Fk} can be terminated
by replacing step $\bf b_k)$ by
\begin{enumerate}

\myitem{$\bf b_k')$} The $q$-generating function $G_k(q)$ is given
by the constant term
$$
G_k(q)\ses F_{\CS\CS_{k-1}}(aq,aq,\ldots ,aq ,q/a,\ldots ,q/a)\Big
|_{a^0}.
$$
\end{enumerate}

The steps up to $\mathbf{b_3)}$ can be carried out by hand. For
further steps we need a computer, and to carry out step
$\mathbf{b_5)}$ by computer we have to introduce one more tool
  as we shall see. Unfortunately Step $\mathbf{b_6)}$ appears beyond reach at the moment.

It will be instructive to see what the first several steps give.

\begin{enumerate}

\myitem {$\bf a_1)$}
$$
F_{\CS_1}(x_1,,x_2)={1\over 1-x_1x_2}.
$$

\myitem {$\bf a_2)$ }
$$
  F_{\CS\CS_1}(x_1,x_2,x_3,x_4)={ (1- x_1 x_2 x_3 x_4 )\over
 (1-x_1 x_2) ( 1- x_2x_3) ( 1-x_1 x_4) ( 1-x_3 x_4)}.
$$
\myitem {$\bf b_2)$}
\begin{align*}
 F_{\CS_2}(x_1,x_2,x_3,x_4) &\ses { (1- x_1 x_2 x_3 x_4 )\over
 (1-a^2x_1 x_2) ( 1- x_2x_3) ( 1-x_1 x_4) ( 1-x_3 x_4/a^2)}\Big|_{a^0}
\cr &\ses { 1\over
 ( 1- x_2x_3) ( 1-x_1 x_4) }.
\end{align*}

\myitem {$\bf a_3)$ } \begin{align*}  F_{\CS\CS_2}(x) &=
{(1-x_1x_4x_5x_8)(1-x_2x_3x_6x_7)\over
(1-x_1x_8)(1-x_2x_7)(1-x_3x_6)(1-x_4x_5)(1-x_1x_4)(1-x_2x_3)(1-x_6x_7)(1-x_5x_8)}
.
\end{align*}

\myitem {$\bf b_3)$ } \begin{align} F_{\CS_3}(x) &=
{(1-x_1x_4x_5x_8)(1-x_2x_3x_6x_7)\over
(1-x_1x_8)(1-x_2x_7)(1-x_3x_6)(1-x_4x_5)} \cr &
\ess\ess\ess\ess\ess\ess\ess\ess\ess\ess\ess\ess\qquad \times{1\over
(1-a^2x_1x_4)(1-a^2x_2x_3)(1-x_6x_7/a^2)(1-x_5x_8/a^2)}\Big|_{a^0}.\ess\ess
\label{3.10}
\end{align}

\end{enumerate}

We can compute this constant term in many ways. In particular we
could use one of the MacMahon identities given by Andrews in [1].
But it is interesting to point out that our divided difference
algorithm has already provided us (in step $\bf a_2)$)  a formula we
can use in step $\bf b_3)$. In fact, the output of step $\bf a_2)$
$$
F_{\CS\CS_1}(x_1,x_2,x_3,x_4)\ses { (1- x_1 x_2 x_3 x_4 )\over
 (1-x_1 x_2) ( 1- x_2x_3) ( 1-x_1 x_4) ( 1-x_3 x_4)}
$$
is the complete generating function of the system
$p_1-p_2+p_3-p_4=0$, so by MacMahon partition analysis we should
also have
$$
 F_{\CS\CS_1}(x_1,x_2,x_3,x_4)\ses{1\over
 (1-ax_1)(1-x_2/a)(1-ax_3)(1-x_4/a)}\Big|_{a^0}.
$$
This implies that
\begin{align*}
 &{1\over (1-a^2x_1x_4)(1-a^2x_2x_3) (1-x_6x_7/a^2)
(1-x_5x_8/a^2)}\bigg|_{a^0} \cr
&\bigsp\bigsp\bigsp\bigsp\ess\ess\ess \ses
 { (1- x_1 x_2 x_3 x_4 )\over
 (1-x_1 x_2) ( 1- x_2x_3) ( 1-x_1 x_4) ( 1-x_3 x_4)}
\Bigg|_{ \mymatrix{ x_1\RA x_1x_4\cr x_3\RA x_2x_3\cr x_2\RA
x_6x_7\cr x_4\RA x_5x_8\cr } } \cr
&\bigsp\bigsp\bigsp\bigsp\ess\ess\ess \ses { (1- x_1 x_2 x_3 x_4x_5
x_6 x_7 x_8  )\over
 (1-x_1x_4 x_6x_7) ( 1- x_6x_7x_2x_3) ( 1-x_1x_4 x_5x_8) ( 1-x_2x_3
 x_5x_8)}.
\end{align*}
Using this in \eqref{3.10} gives
\begin{align*} F_{\CS_3}(x_1,\ldots
,x_8) &\ses {(1-x_1x_4x_5x_8)(1-x_2x_3x_6x_7)\over
(1-x_1x_8)(1-x_2x_7)(1-x_3x_6)(1-x_4x_5)} \cr &
\ess\ess\ess\ess\ess\ess\ess\ess\ess\ess\ess\ess\ess\ess\ess\ess\ess\ess\ess
\times { (1- x_1 x_2 x_3 x_4x_5 x_6 x_7 x_8  )\over
 (1-x_1x_4 x_6x_7) ( 1- x_6x_7x_2x_3) ( 1-x_1x_4 x_5x_8) ( 1-x_2x_3 x_5x_8)}
\cr &\ses { 1- x_1 x_2 x_3 x_4x_5 x_6 x_7 x_8  \over
(1-x_1x_8)(1-x_2x_7)(1-x_3x_6)(1-x_4x_5)(1-x_1x_4 x_6x_7)     (
1-x_2x_3 x_5x_8)} .
\end{align*}
Replacing all the $x_i$ by the single variable $q$, we thus obtain
that
$$
G_1(q)={1\over 1-q^2}\scs\ess  G_2(q)={1\over (1-q^2)^2}\scs\ess
 G_3(q)={1-q^8\over (1-q^2)^4(1-q^4)^2}
\ses{1+q^4\over (1-q^2)^4(1-q^4) }. 
$$
Using the computer to carry out step $\bf b_4')$ gives
$$
 G_4(q)=
  {1+q^2 +21 q^4+36q^6+74q^8+86q^{10}+74q^6+36q^{14}+21q^{16}+q^{18}+q^{20}\over(1-q^2)^7(1-q^4)^4(1-q^6) }.
$$
We shall see later what else has to be done to obtain $G_5(q)$.

\medskip

Our divided difference algorithm can also be adapted to compute the
first 4 Sdd series as well. In fact, again due to the fact that
divided difference operators commute with the constant term
operators, we can also show  that  all the complete Sdd series can
(in principle) be obtained by the following algorithm.

\begin{Algorithm}[Sdd Case]
\label{algo-Wk}

\begin{enumerate}
\myitem {$\bf a_1)$} Initially compute the complete generating
function for the Sdd problem for $k=1$. That is, compute the
constant term
$$
W_1(x_1,x_2)={1-a^2\over (1-  x_1a)(1-x_2/a)}\ssp \Big|_{a^0}.
$$

\myitem{$\bf a_k)$} With $W_{k-1}(x_1,\ldots ,x_{2^{k-1}})$ from
step $\bf b_{k-1})$, compute by divided difference
$$
WW_{{k-1}}(x_1,\ldots ,x_{2^{k}})= \dd_{1,1+2^{k-1}}\cdots
\dd_{2^{k-1},2^{k}}W_{k-1}(x_1,\ldots ,x_{2^{k-1}}).
$$

\myitem{$\bf b_k )$} With $WW_{k-1}(x_1,\ldots ,x_{2^{k-1}})$ from
step $\bf a_{k})$,  compute the complete generating function for the
Sdd problem for $k$ by  the following constant term:
$$
W_k(x_1,x_2,\ldots,x_{2^k})= WW_{k-1}(ax_1,ax_2,\ldots
,ax_{2^{k-1}},x_{2^{k-1}+1}/a,\ldots ,x_{2^k}/a)(1-a^2)\Big |_{a^0}.
$$
\end{enumerate}
\end{Algorithm}

Note that similarly as for the Hdd-case, the sequence of steps in
Algorithm \ref{algo-Wk} can be terminated by replacing step $\bf
b_k)$ by
\begin{enumerate}
\myitem{$\bf b_k')$} To obtain the  generating function $W_k(q)$
compute the constant term
$$
W_k(q)= WW_{k-1}(aq,aq,\ldots ,aq ,q/a,\ldots ,q/a)(1-a^2)\Big
|_{a^0}.
$$
\end{enumerate}

Only  steps $\bf a_1)$ and $\bf a_2)$ can be carried out by hand.
Though steps 3 and 4 are routine they are too messy to do by hand.
But step 5 again needs further tricks to be carried out by computer.
Step 6 appears beyond reach at the moment.

It will be instructive to see what some of these steps give.
\begin{enumerate}
\myitem {$\bf a_1)$}
$$
W_1 (x_1,x_2)={1-x_2^2\over 1-x_1x_2}.
$$

\myitem  {$\bf a_2)$ }
$$
WW_1 (x_1,\ldots,x_4)= {1-x_2^2 -x_2x_4 -x_4^2 +x_1x_2^2x_4 + x_2^2
x_3 x_4 -x_1 x_2x_3 x_4 +x_1 x_2 x_4^2 +x_2 x_3 x_4^2 - x_1 x_2^2
x_3x_4^2 \over
 ( 1 - x_1 x_2) ( 1 - x_3 x_2) ( 1 - x_1 x_4) ( 1 - x_3 x_4)}.
$$

\myitem  {$\bf b_2)$ }
$$
W_2(x_1,x_2,x_3,x_4)= {1-x_2x_4-x_3x_4+x_4^2\over
(1-x_1x_4)(1-x_2x_3)}.
$$
This gives
$$
W_2(q)= {1\over 1-q^2 }.
$$

\myitem  {$\bf a_3)$ }
$$
WW_2(x_1,\ldots,x_8)={(large\ess\ess numerator)\over (1-x_1x_4)
(1-x_1x_8) (1-x_2x_3) (1-x_2x_7) (1-x_3x_6) (1-x_4x_5) (1-x_5x_8)
(1-x_6x_7) }.
$$

\myitem  {$\bf b_3 )$ }
$$
W_3(x_1,\ldots,x_8)= {(large\ess\ess numerator)\over (1-x_1x_8)
(1-x_2x_7) (1-x_3x_6) (1-x_4x_5) (1-x_1x_4x_6x_7) (1-x_2x_3x_5x_8)
}.
$$


\myitem  {$\bf b_3')$} Notwithstanding the complexity of the
previous results it turns out that to obtain $W_3(q)$ we need only
compute the constant term
\begin{align}
W_3(q)= {1\over (1-q^2)}\times {1-a^2\over
(1-q^2a^2)(1-q^2/a^2)}\bigg|_{a^0}. \label{3.12}
\end{align}
To this end we start by determining the coefficients $A$ and $B$ in the partial fraction decomposition
$$
{(1-a^2)a^2\over (1-q^2a^2)(a^2-q^2 )}= {1\over q^2}\sps {A\over
1-q^2a^2} +{B\over a^2-q^2 }
$$
obtaining
\begin{align*}
A&= {(1-a^2)a^2\over  (a^2-q^2 )}\bigg|_{a^2=1/q^2}= \,
{(1-1/q^2)/q^2\over  (1/q^2-q^2 )} = -{1\over q^2(1+q^2)},\\
B&= {(1-a^2)a^2\over (1-q^2a^2)}\bigg|_{a^2=q^2}= {(1-q^2)q^2\over
(1-q^4)}= { q^2\over (1+q^2)},
\end{align*}
(the exact value of $B$ is not needed) and we can write
$$
{1-a^2\over (1-q^2a^2)(1-q^2/a^2)}= {1\over q^2}\sms {1\over
q^2(1+q^2)}\times {1\over (1-a^2q^2)} \sps {1\over
(1+q^2)}\times{q^2/a^2\over 1-q^2/a^2 }.
$$
Thus taking constant terms gives
$$
{1-a^2\over (1-q^2a^2)(1-q^2/a^2)}\bigg|_{a^0}= {1\over q^2}\sms
{1\over q^2(1+q^2)}\sps 0 = {1\over 1+q^2}.
$$
Using this in \eqref{3.12} we finally obtain
$$
W_3(q)={1\over 1-q^4}.
$$

\myitem  {$\bf a_4)$ }
$$
WW_4(x_1,x_2,\ldots, x_{16})= (too\ess  large\ess for\ess
typesetting )
$$

\myitem  {$\bf b_4')$ } Notwithstanding the complexity of the
previous result
 it turns out that to obtain $W_4(q)$ we need only compute the constant term
$$
W_4(q)= {(1+q^4)(1+q^6)\over(1-q^2)(1-q^4)^2} \times {1-a^2\over
(1-a^2q^4)(1-q^4/a^2 )(1-a^4q^4)(1-q^4/a^4 ) }\bigg|_{a^0}.
$$

\end{enumerate}

To illustrate the power and flexibility of the partial fraction
algorithm we will carry this out by hand. The reader is referred to
\cite{2} for a brief tutorial on the use of this algorithm. In the
next few lines we will strictly adhere to the notation and
terminology given  in \cite{2}.

To begin  we note that we need only calculate the constant term
\begin{align}
 C(x)= {1-a \over (1-a x)(1-x/a )(1-a^2x)(1-x/a^2 )
}\bigg|_{a^0}, \label{3.14}
\end{align}
since we can write
\begin{align}
 W_4(q)= {(1+q^4)(1+q^6)\over(1-q^2)(1-q^4)^2} \times C(q^4).
\label{3.15}
\end{align}
Now we have
$$
{1  \over  (1-a^2x)(1-x/a^2 )}= { a^2 \over  (1-a^2x)(a^2-x  )} =
{1\over 1-x^2}{1\over 1-a^2x}\sps {1\over 1-x^2}{x/a^2\over
1-x/a^2}.
$$
Thus \eqref{3.14} may be rewritten in the  form
\begin{align}  C(x) &=
{1\over 1-x^2}\left( {(1-a) \over(1-a x)(1-x/a ) }{1\over
1-a^2x} \bigg|_{a^0} \sps {(1-a) \over(1-a x)(1-x/a  ) } {x/a^2\over
1-x/a^2}\bigg|_{a^0}\right ).\label{3.16}
\end{align}
Note that in the first constant term we have only one dually
contributing term and on the second we have only one contributing
term. This gives
\begin{align}
{(1-a) \over(1-a x)(1-x/a  ) }{1\over 1-a^2x} \bigg|_{a^0}&= {(1-a)
\over(1-a x)  }{1\over 1-a^2x}\bigg|_{a=x} = {(1-x)
\over(1-x^2)  }{1\over 1-x^3} \label{3.17} \\
 {(1-a) \over(1-a x)(1-x/a  ) } {x/a^2\over 1-x/a^2}\bigg|_{a^0}
&= {(1-a) \over (1-x/a  ) } {x/a^2\over 1-x/a^2}\bigg|_{a=1/x} =
{-(1- x) \over (1-x^2  ) } {x^2\over 1-x^3}. \label{3.18}
\end{align}
Using \eqref{3.17} and \eqref{3.18} in \eqref{3.16} we get
\begin{align*}
 C(x) =
{1\over 1-x^2} \left({(1-x) \over(1-x^2)  } {1\over 1-x^3}\sms  {
(1- x) \over (1-x^2  ) } {x^2\over 1-x^3} \right ) ={1-x\over
(1-x^2) (1-x^3)} .
\end{align*}
Together with \eqref{3.15}, we get
$$
W_4(q) = {(1+q^4)(1+q^6)\over(1-q^2)(1-q^4)^2} \times  {1-q^4\over
(1-q^8) (1-q^{12})} =
  {1\over (1-q^2) (1-q^4)^2 (1-q^{6})}.
$$
We will see in section 4  what needs to be done to carry out step  $\bf b'_5$
on the computer.
\sa

The identities for $W_2(q),W_3(q),W_4(q)$ in \eqref{3.20} have also
been derived  in \cite{2} by symmetric function methods from the
relation \eqref{I.8}.
In fact, all three results in \eqref{3.20} are immediate
consequences of the following deeper symmetric function  identity.
(for a proof see \cite[Section 2]{2}.)

\begin{Theorem}
\label{t-3.3}
$$
s_{d,d}\odot s_{d,d}= \sum_{\la\part 2d}s_\la \,\chi(\la  \in EO_4)
$$
where $EO_4$ denotes the set of partitions of length $4$ whose parts
are $\ge 0$ and all even or all odd.
\end{Theorem}

Note that the Kronecker product identity
$$
 \LL s_{d,d} \odot s_{d,d} \odot s_{d,d} \odot s_{d,d}  \odot s_{d,d} \scs s_{2d}\RR
\ses\LL s_{d,d} \odot s_{d,d} \odot s_{d,d}   \scs  s_{d,d}  \odot
s_{d,d}\RR.
$$
suggests obtaining $W_5(q)$ by means of a combinatorial
interpretation of the coefficients of the Schur function expansion
of the Kronecker product $s_{d,d} \odot s_{d,d} \odot s_{d,d}$.
However, to this date no formula has been given for these
coefficients, combinatorial or otherwise.


\section{\label{sec-4} Solving the Hdd problem  for $k=5$  }
This section is divided into four parts. In the first subsection we
start with our computer findings and end by giving a combinatorial
decomposition that works nicely to obtain $F_3(x)$. In the second
subsection, this decomposition is described algebraically and,
together with group actions, turned into manipulatory gyrations that
will be used to extract $G_5(q)$ and $W_5(q)$ out of our computers.
In the third subsection, by combining the idea of decomposition and
the method of divided difference in Section \ref{sec-3}, we give our
best way that reduce the computation time for $G_5(q)$ and $W_5(q)$
down to a few minutes. In the final subsection, we give our first
algorithm to obtain $G_5(q)$ and $W_5(q)$.

\subsection{A combinatorial decomposition for $F_3(x)$}
Our initial efforts at solving the Hdd an Sdd  problems were
entirely carried out by computer experimentation. After obtaining
quite easily the series $G_2(q)$,  $G_3(q)$, $G_4(q)$ and $W_2(q)$,
$W_3(q)$, $W_4(q)$, all the computer packages available to us failed
to directly  deliver $G_5(q)$ and $W_5(q)$.

The computer data  obtained for the Hdd problem for $k=2,3$ were
combinatorially so revealing that we have been left with a strong
impression that this problem should have a very beautiful
combinatorial general solution. Only  time will tell if this will
ever be the case.  To stimulate further research we will begin  by
reviewing our initial computer and manual combinatorial findings.

Recall that we denoted by $\CF_d$ the collection of all $d$-subsets
of the $2d$ element set $\OM_{2d}$. We also showed (in Theorem
\ref{t-2.1}) that the coefficient $m_d(k)$ in the series $
G_k(q)=\sum_{d\le 0} q^{2d}m_d(k)
$
counts the number of orbits under the action of the symmetric group
$\CS_{2d}$ on the $k$-fold  cartesian product $\CF_d\times
\CF_d\times \cdots \times \CF_d$. Denoting by $(A_1,A_2 ,\ldots,
A_k)$ a  generic element of this cartesian product, then each orbit
is uniquely determined by the $2^k$ cardinalities
$$
p_{\eee_1,\eee_2,\cdots ,\eee_k}= \big|A_1^{\eee_1}\cap
A_2^{\eee_2}\cap\cdots \cap A_k^{\eee_k}\big|
$$
where for each $1\le i\le k$ we set

\newcommand\mycases[1]{\left\{\begin{matrix}
#1
\end{matrix}\right.}
$$
A^{\eee_i}_i= \left\{ \begin{array}{ll} \,\,\,A_i & \textrm{ if }
\eee_i=0, \cr \ ^cA_i & \textrm{ if } \eee_i=1. \end{array}\right.
\bigsp (\hbox{here $\ ^cA_i=\OM_{2d}/A_i $}).
$$
It is also convenient to set $ A_{\eee_1,\eee_2,\cdots ,\eee_k}=
A_1^{\eee_1}\cap
A_2^{\eee_2}\cap\cdots \cap A_k^{\eee_k}. 
$
This given we have seen that the condition $(A_1,A_2 ,\ldots,
A_k)\in \CF_d^k $ is equivalent to the Diophantine system
$$
\CS_k= \left\|\ \mymatrix{ \sum_{\eee_1=0}^1\sum_{\eee_2=0}^1\cdots
\sum_{\eee_k=0}^1 (1-2\eee_1)p_{\eee_1,\eee_2,\cdots ,\eee_k}= 0,\cr
\sum_{\eee_1=0}^1\sum_{\eee_2=0}^1\cdots \sum_{\eee_k=0}^1
(1-2\eee_2)p_{\eee_1,\eee_2,\cdots ,\eee_k}= 0,\cr \vdots
\ess\ess\ess\ess \ess\ess\ess\vdots \ess\ess\ess\ess
\ess\ess\ess\vdots \ess\ess\ess\ess \ess\ess\ess \vdots
\ess\ess\ess\ess \ess\ess\ess\vdots \ess\ess\ess\ess \ess\ess\ess
\ess\ess \vdots \cr \sum_{\eee_1=0}^1\sum_{\eee_2=0}^1\cdots
\sum_{\eee_k=0}^1 (1-2\eee_k)p_{\eee_1,\eee_2,\cdots ,\eee_k}=
0,\cr }\right . 
$$
together with the condition $|\OM_{2d}|=2d$, that is $
\sum_{\eee_1=0}^1\sum_{\eee_2=0}^1\cdots \sum_{\eee_k=0}^1
p_{\eee_1,\eee_2,\cdots ,\eee_k}= 2d.
$

There are several algorithms available to solve such a system. See
for instance \cite[Chapter~4.6]{6}. 
 The algorithm we used for our computer
experimentations is the  MacMahon algorithm which has been recently
implemented in MATHEMATICA by   Andrews, Paule and Riese and in
MAPLE by Xin using the partial fraction method of computing constant
terms.

 \noindent The former
can be downloaded from the web site

\centerline {\ita http://www.risc.uni-linz.ac.at/research/combinat/software/Omega/}

\noindent and the latter from the web site

\centerline {\ita
http://www.combinatorics.net.cn/homepage/xin/maple/ell2.rar.}

\noindent For computer implementation we found it more convenient to
use the alternate notation adopted in Remark \ref{rem-1.3}. That is
\begin{align}
 \CS_k= \left\|\ p_1V_1\sps  p_2V_2 \sps \cdots \sps
p_{2^k}V_{2^k}= 0\right. .\label{4.5}
\end{align}
These algorithms may yield quite a bit more than the number of solutions of such a system.
For instance, in our case
letting $\CC_k$ denote the collection of solutions of the system $\CS_k$, the ``Omega package'' of
Andrews, Paule and Riese should, in principle, yield the formal power series
$$
F_k(x_1,x_2,\ldots,x_{2^k})\, = \sum_{(p_1,p_2,\ldots ,p_{2^k})\in \, \CC_k} x_1^{p_1}x_2^{p_2}\cdots x_{2^k}^{p_{2^k}}.
$$
It follows from the general theory of Diophantine systems that
$F_k(x_1,x_2,\ldots,x_{2^k})$ is always the Taylor series of a
rational function.

Now for $\CS_2$ and $\CS_3$ the Omega package gives
\begin{align}
 F_2(x_1,x_2,x_3,x_4)&= {1\over (1-x_1x_4)(1-x_2x_3)} \label{4.7}\\
F_3(x_1,x_2,\ldots,x_8)&={1-x_2x_3x_5x_8x_1x_4x_6x_7\over
(1-x_1x_8)(1-x_2x_7)(1-x_3x_6)(1-x_4x_5)(1-x_2x_3x_5x_8)(1-x_1x_4x_6x_7)}.
\label{4.8}
\end{align}
But this is as far as this package went in our computers.
However we could go further by giving up full information about the solutions
 and only ask for the series
$$
G_k(q)= F_k(x_1,x_2,\ldots,x_{2^k})\big|_{x_i=q},
$$
which can be computed from its constant term representation in
\eqref{I.12}.
For example, the program {\ita Latte} by De Loera, Hemmecke, Tauzer,
Yoshida, which is available at

\centerline{\ita http://www.math.ucdavis.edu /\~\ latte/}

\noindent computed the $G_4(q)$ series in approximately $30$
seconds.  However, this is as far as {\tt Latte} went on our
machines. We  should also mention that all the series $G_k(q)$ and
$W_k(q)$ for $k\le 4$  can be obtained
 in only a few seconds, from the software of Xin by computing the corresponding constant terms
in \eqref{I.12} and \eqref{I.13}.

To get our  computers to deliver $G_5(q)$ and $W_5(q)$ in a matter
of  minutes a divide and conquer strategy had to be adopted. More
precisely, these rational functions were obtained  by decomposing
the constant terms \eqref{I.12} and \eqref{I.13} as  sums of
constant terms. This decomposition had its origin from an effort to
find a human proof of the identities  in \eqref{4.7} and
\eqref{4.8}. More importantly, the surprising simplicity of
\eqref{4.7} and \eqref{4.8} required a combinatorial explanation.
Our findings there provided the combinatorial tools that were used
in our early computations of $G_5(q)$ and $W_5(q)$. This given,
before describing our work on these series, we will show how to
obtain \eqref{4.7} and \eqref{4.8} entirely by hand.

\medskip
Let us start by sketching the idea for $k=2$. Beginning with
$$
\CS_2= \left\| \eqalign { &p_1+p_2 -  p_3-p_4=   0\cr & p_1 -
p_2+p_3-p_4   =  0\cr }\right .
$$
we immediately notice that $(1,0,0,1) $ and $(0,1,1,0)$
are solutions. Set
$$ a=\min(p_1,p_4) \text{ and } b=\min(p_2,p_3).$$ It is
clear that the following difference must also be a solution.
$$
(q_1,q_2,q_3,q_4)=
(p_1,p_2,p_3,p_4)-(a,b,b,a)=(p_1-a,p_2-b,p_3-b,p_4-a).
$$
Now $q_1q_4=0$ and $q_2q_3=0$. This gives us four possibilities for
$(q_1,q_2,q_3,q_4)$:
\begin{align}
 (0,0,x,y)\scs \ess (0,x,0,y)\scs \ess (x,0,y,0)\scs \ess
(x,y,0,0), \label{4.13}
\end{align}
for some nonnegative integers $x,y$. Testing the first equation of
$\CS_2$  immediately forces the first and  last in \eqref{4.13}  to
identically vanish. Similarly, the second equation of $\CS_2$ yields
that the second and third in  \eqref{4.13} must also identically
vanish. This proves that the general solution of $\CS_2$ is of the
form $(a,b,b,a)$. We thus reobtain the full generating function
\eqref{4.7} of solutions  of $\CS_2$:
$$
F_2(x_1,x_2,x_3,x_4)= \sum_{a\ge 0}\sum_{b\ge
0}x_1^ax_2^bx_3^bx_4^a= {1\over (1-x_1x_4)(1-x_2x_3)}.
$$

\medskip
It turns out that we can deal with $\CS_3$ in a similar manner.
Again we begin by noticing the four \emph{ symmetric} solutions
$$
(1,0,0,0,0,0,0,1)\scs \ess (0,1,0,0,0,0,1,0)\scs \ess (0,0,1,0,0,1,0,0) \scs \ess (0,0,0,1,1,0,0,0).
$$
Next we set
$$
a=\min(p_1,p_8), \ess b=\min(p_2,p_7), \ess c=\min(p_3,p_6), \ess
d=\min(p_4,p_5), \ess
$$
and by subtraction we get a solution
\begin{align}
 (q_1,q_2,q_3,q_4,q_5,q_6,q_7,q_8)=
(p_1,p_2,p_3,p_4,p_5,p_6,p_7,p_8)-(a,b,c,d,d,c,b,a) \label{4.14}
\end{align}
with the property $q_iq_{9-i}=0$ for $1\le i\le 4$. It will be good
here and after to call the set
$$
\{ i\in [1,n]\,:\, p_i\ge 1 \}
$$
the \emph{ support} of the composition $(p_1,p_2,\ldots ,p_n)$. This
given, we derive that the resulting composition in \eqref{4.14} will
necessarily have its support contained in at least one of the
following 16 patterns.
\begin{align}
 \eqalign{
&(0,0,0,0,*,*,*,*)\scs\,(0,0,0,*,0,*,*,*)\scs\,(0,0,*,0,*,0,*,*)\scs\,(0,0,*,*,0,0,*,*)\scs\cr
&(0,*,0,0,*,*,0,*)\scs\,(0,*,0,*,0,*,0,*)\scs\,(0,*,*,0,*,0,0,*)\scs\,(0,*,*,*,0,0,0,*)\scs\cr
&(*,0,0,0,*,*,*,0)\scs\,(*,0,0,*,0,*,*,0)\scs\,(*,0,*,0,*,0,*,0)\scs\,(*,0,*,*,0,0,*,0)\scs\cr
&(*,*,0,0,*,*,0,0)\scs\,(*,*,0,*,0,*,0,0)\scs\,(*,*,*,0,*,0,0,0)\scs\,(*,*,*,*,0,0,0,0).
} \label{4.16}
\end{align}
\noindent Unlike the case $k=2$ not all of these patterns force a
trivial solution. To find out which it is helpful to resort to a
Venn diagram imagery. To this end recall that a solution of $\CS_3$
gives the cardinalities of  the 8 regions of the Venn  diagram  of
three $d$-subsets $A_1,A_2,A_3$  of $\Omega_{2d}$ (see Figure
\ref{fig-S3}).

\begin{figure}[hbt]
\begin{center}
\input{ven3.pstex_t}
\end{center}


 \caption{The Ven
diagram for $\CS_3$. \label{fig-S3}}
\end{figure}

 \noindent In Figure \ref{fig-16},  each pattern is represented
by a Venn diagram where in each region
 $A_1^{\eee_1}\cap A_2^{\eee_2}\cap A_3^{\eee_3}$ that corresponds to a $*$ in the pattern
we placed a black dot. That means that only the regions with a dot may have $\ge 0$ cardinality.
 The miracle is that
all but the two patterns
$(0,*,*,0,*,0,0,*)$ and $(*,0,0,*,0,*,*,0)$ can be quickly excluded by a reasoning that only
uses the positions of the dots in the Venn diagram.
 In fact, in each of the excluded cases, we show that it
is impossible to replace the dots by $\ge 0$ integers in such a
manner that the three sets $A_1,A_2,A_3$ and their complements $\
^cA_1,\ ^cA_2,\ ^c A_3$ end up having the same cardinality (except
for all empty sets).

\begin{figure}[hbt]
\begin{center}
\begin{picture}(400,68) \put(200,0){ \put(0,0){\circle{40}}
\put(0,-7){\circle{21}}
\put(-5.5,2){\circle{21}}\put(5.5,2){\circle{21}}
\put(-1,-2.5){\tiny{1}}\put(-1,4.5){\tiny{2}}
\put(-7,-5.5){\tiny{3}}\put(-10,4.5){\tiny{4}}
\put(4.5,-5.5){\tiny{5}}\put(7,4.5){\tiny{6}}
\put(-1,-11.5){\tiny{7}}\put(-1,14.5){\tiny{8}}
\put(0,-3){\circle*{2}} \put(0,4.5){\circle*{2}}
\put(8,-4.5){\circle*{2}} \put(12,3.5){\circle*{2}}

\put(50,0){\put(0,0){\circle{40}} \put(0,-7){\circle{21}}
\put(-5.5,2){\circle{21}}\put(5.5,2){\circle{21}}
\put(-1,-2.5){\tiny{1}}\put(-1,4.5){\tiny{2}}
\put(-7,-5.5){\tiny{3}}\put(-10,4.5){\tiny{4}}
\put(4.5,-5.5){\tiny{5}}\put(7,4.5){\tiny{6}}
\put(-1,-11.5){\tiny{7}}\put(-1,14.5){\tiny{8}}
\put(0,-3){\circle*{2}} \put(0,4.5){\circle*{2}}
\put(-12,3.5){\circle*{2}} \put(12,3.5){\circle*{2}} }

\put(100,0){\put(0,0){\circle{40}} \put(0,-7){\circle{21}}
\put(-5.5,2){\circle{21}}\put(5.5,2){\circle{21}}
\put(-1,-2.5){\tiny{1}}\put(-1,4.5){\tiny{2}}
\put(-7,-5.5){\tiny{3}}\put(-10,4.5){\tiny{4}}
\put(4.5,-5.5){\tiny{5}}\put(7,4.5){\tiny{6}}
\put(-1,-11.5){\tiny{7}}\put(-1,14.5){\tiny{8}}
\put(0,-3){\circle*{2}} \put(0,4.5){\circle*{2}}
\put(-8,-4.5){\circle*{2}} \put(8,-4.5){\circle*{2}}

}

\put(150,0){\put(0,0){\circle{40}} \put(0,-7){\circle{21}}
\put(-5.5,2){\circle{21}}\put(5.5,2){\circle{21}}
\put(-1,-2.5){\tiny{1}}\put(-1,4.5){\tiny{2}}
\put(-7,-5.5){\tiny{3}}\put(-10,4.5){\tiny{4}}
\put(4.5,-5.5){\tiny{5}}\put(7,4.5){\tiny{6}}
\put(-1,-11.5){\tiny{7}}\put(-1,14.5){\tiny{8}}
\put(0,-3){\circle*{2}}\put(0,4.5){\circle*{2}}
\put(-8,-4.5){\circle*{2}} \put(-12,3.5){\circle*{2}} }
}

\put(0,0){ \put(0,0){\circle{40}} \put(0,-7){\circle{21}}
\put(-5.5,2){\circle{21}}\put(5.5,2){\circle{21}}
\put(-1,-2.5){\tiny{1}}\put(-1,4.5){\tiny{2}}
\put(-7,-5.5){\tiny{3}}\put(-10,4.5){\tiny{4}}
\put(4.5,-5.5){\tiny{5}}\put(7,4.5){\tiny{6}}
\put(-1,-11.5){\tiny{7}}\put(-1,14.5){\tiny{8}}
\put(0,-3){\circle*{2}} \put(8,-4.5){\circle*{2}}
\put(12,3.5){\circle*{2}} \put(0,-14.5){\circle*{2}}

\put(50,0){\put(0,0){\circle{40}} \put(0,-7){\circle{21}}
\put(-5.5,2){\circle{21}}\put(5.5,2){\circle{21}}
\put(-1,-2.5){\tiny{1}}\put(-1,4.5){\tiny{2}}
\put(-7,-5.5){\tiny{3}}\put(-10,4.5){\tiny{4}}
\put(4.5,-5.5){\tiny{5}}\put(7,4.5){\tiny{6}}
\put(-1,-11.5){\tiny{7}}\put(-1,14.5){\tiny{8}}
\put(0,-3){\circle*{2}}
 \put(-12,3.5){\circle*{2}}
 \put(12,3.5){\circle*{2}}
\put(0,-14.5){\circle*{2}} }

\put(100,0){\put(0,0){\circle{40}} \put(0,-7){\circle{21}}
\put(-5.5,2){\circle{21}}\put(5.5,2){\circle{21}}
\put(-1,-2.5){\tiny{1}}\put(-1,4.5){\tiny{2}}
\put(-7,-5.5){\tiny{3}}\put(-10,4.5){\tiny{4}}
\put(4.5,-5.5){\tiny{5}}\put(7,4.5){\tiny{6}}
\put(-1,-11.5){\tiny{7}}\put(-1,14.5){\tiny{8}}
\put(0,-3){\circle*{2}} \put(-8,-4.5){\circle*{2}}
\put(8,-4.5){\circle*{2}} \put(0,-14.5){\circle*{2}} }

\put(150,0){\put(0,0){\circle{40}} \put(0,-7){\circle{21}}
\put(-5.5,2){\circle{21}}\put(5.5,2){\circle{21}}
\put(-1,-2.5){\tiny{1}}\put(-1,4.5){\tiny{2}}
\put(-7,-5.5){\tiny{3}}\put(-10,4.5){\tiny{4}}
\put(4.5,-5.5){\tiny{5}}\put(7,4.5){\tiny{6}}
\put(-1,-11.5){\tiny{7}}\put(-1,14.5){\tiny{8}}
\put(0,-3){\circle*{2}} \put(-8,-4.5){\circle*{2}}
\put(-12,3.5){\circle*{2}}

\put(0,-14.5){\circle*{2}} }

}

\put(200,50){\put(0,0){\circle{40}} \put(0,-7){\circle{21}}
\put(-5.5,2){\circle{21}}\put(5.5,2){\circle{21}}
\put(-1,-2.5){\tiny{1}}\put(-1,4.5){\tiny{2}}
\put(-7,-5.5){\tiny{3}}\put(-10,4.5){\tiny{4}}
\put(4.5,-5.5){\tiny{5}}\put(7,4.5){\tiny{6}}
\put(-1,-11.5){\tiny{7}}\put(-1,14.5){\tiny{8}}
 \put(0,4.5){\circle*{2}}

\put(8,-4.5){\circle*{2}} \put(12,3.5){\circle*{2}}
 \put(4,16.5){\circle*{2}}

\put(50,0){\put(0,0){\circle{40}} \put(0,-7){\circle{21}}
\put(-5.5,2){\circle{21}}\put(5.5,2){\circle{21}}
\put(-1,-2.5){\tiny{1}}\put(-1,4.5){\tiny{2}}
\put(-7,-5.5){\tiny{3}}\put(-10,4.5){\tiny{4}}
\put(4.5,-5.5){\tiny{5}}\put(7,4.5){\tiny{6}}
\put(-1,-11.5){\tiny{7}}\put(-1,14.5){\tiny{8}}

 \put(0,4.5){\circle*{2}}
 \put(-12,3.5){\circle*{2}}
 \put(12,3.5){\circle*{2}}
 \put(4,16.5){\circle*{2}}
}

\put(100,0){\put(0,0){\circle{40}} \put(0,-7){\circle{21}}
\put(-5.5,2){\circle{21}}\put(5.5,2){\circle{21}}
\put(-1,-2.5){\tiny{1}}\put(-1,4.5){\tiny{2}}
\put(-7,-5.5){\tiny{3}}\put(-10,4.5){\tiny{4}}
\put(4.5,-5.5){\tiny{5}}\put(7,4.5){\tiny{6}}
\put(-1,-11.5){\tiny{7}}\put(-1,14.5){\tiny{8}}
\put(0,4.5){\circle*{2}} \put(-8,-4.5){\circle*{2}}
\put(8,-4.5){\circle*{2}}
 \put(4,16.5){\circle*{2}}
} \put(150,0){\put(0,0){\circle{40}} \put(0,-7){\circle{21}}
\put(-5.5,2){\circle{21}}\put(5.5,2){\circle{21}}
\put(-1,-2.5){\tiny{1}}\put(-1,4.5){\tiny{2}}
\put(-7,-5.5){\tiny{3}}\put(-10,4.5){\tiny{4}}
\put(4.5,-5.5){\tiny{5}}\put(7,4.5){\tiny{6}}
\put(-1,-11.5){\tiny{7}}\put(-1,14.5){\tiny{8}}

\put(0,4.5){\circle*{2}} \put(-8,-4.5){\circle*{2}}
\put(-12,3.5){\circle*{2}}
 \put(4,16.5){\circle*{2}}
}

} \put(0,50){\put(0,0){\circle{40}} \put(0,-7){\circle{21}}
\put(-5.5,2){\circle{21}}\put(5.5,2){\circle{21}}
\put(-1,-2.5){\tiny{1}}\put(-1,4.5){\tiny{2}}
\put(-7,-5.5){\tiny{3}}\put(-10,4.5){\tiny{4}}
\put(4.5,-5.5){\tiny{5}}\put(7,4.5){\tiny{6}}
\put(-1,-11.5){\tiny{7}}\put(-1,14.5){\tiny{8}}
\put(8,-4.5){\circle*{2}} \put(12,3.5){\circle*{2}}
\put(0,-14.5){\circle*{2}} \put(4,16.5){\circle*{2}}
\put(50,0){\put(0,0){\circle{40}} \put(0,-7){\circle{21}}
\put(-5.5,2){\circle{21}}\put(5.5,2){\circle{21}}
\put(-1,-2.5){\tiny{1}}\put(-1,4.5){\tiny{2}}
\put(-7,-5.5){\tiny{3}}\put(-10,4.5){\tiny{4}}
\put(4.5,-5.5){\tiny{5}}\put(7,4.5){\tiny{6}}
\put(-1,-11.5){\tiny{7}}\put(-1,14.5){\tiny{8}}
\put(-12,3.5){\circle*{2}} \put(12,3.5){\circle*{2}}
\put(0,-14.5){\circle*{2}} \put(4,16.5){\circle*{2}} }

\put(100,0){\put(0,0){\circle{40}} \put(0,-7){\circle{21}}
\put(-5.5,2){\circle{21}}\put(5.5,2){\circle{21}}
\put(-1,-2.5){\tiny{1}}\put(-1,4.5){\tiny{2}}
\put(-7,-5.5){\tiny{3}}\put(-10,4.5){\tiny{4}}
\put(4.5,-5.5){\tiny{5}}\put(7,4.5){\tiny{6}}
\put(-1,-11.5){\tiny{7}}\put(-1,14.5){\tiny{8}}

\put(-8,-4.5){\circle*{2}} \put(8,-4.5){\circle*{2}}
\put(0,-14.5){\circle*{2}} \put(4,16.5){\circle*{2}} }

\put(150,0){\put(0,0){\circle{40}} \put(0,-7){\circle{21}}
\put(-5.5,2){\circle{21}}\put(5.5,2){\circle{21}}
\put(-1,-2.5){\tiny{1}}\put(-1,4.5){\tiny{2}}
\put(-7,-5.5){\tiny{3}}\put(-10,4.5){\tiny{4}}
\put(4.5,-5.5){\tiny{5}}\put(7,4.5){\tiny{6}}
\put(-1,-11.5){\tiny{7}}\put(-1,14.5){\tiny{8}}
\put(-8,-4.5){\circle*{2}} \put(-12,3.5){\circle*{2}}
\put(0,-14.5){\circle*{2}} \put(4,16.5){\circle*{2}} } }
\end{picture}
\end{center}

\vspace{5mm} \caption{The 16 support patterns for $\CS_3$.
\label{fig-16}}
\end{figure}
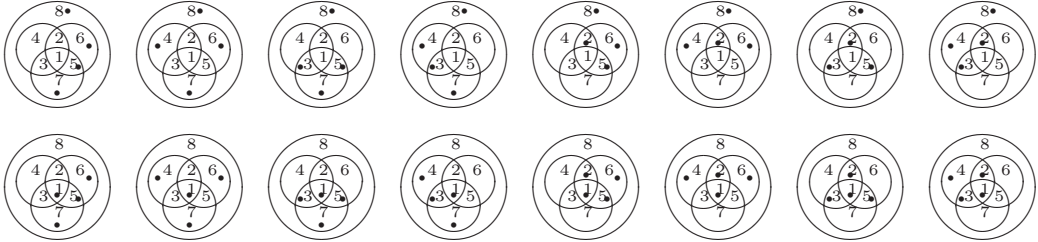

The reasoning is so cute that we are compelled to present it here in
full. In what follows the $j^{\mathrm{th}}$ diagram in the
$i^{\mathrm{th}}$ row will be referred to as ``$D_{ij}$'':

\begin{enumerate}
\myitem {(1)} $D_{11}$, $D_{14}$, $D_{16}$, $D_{23}$, $D_{25}$, and
$D_{28}$ can be immediately excluded because one of $A_1$, $A_2$,
$A_3$, $A_1^c$, $A_2^c$ or $A_3^c$ would be empty.

\myitem {(2)} In $D_{15}$ the dot next to 8 should give the
cardinality of $A_2^c$ (say $d$) and then the dot next to the 2
should also give $d$. But that forces the dots next to 5 and 6 to be
$0$, leaving $A_3$ empty, a contradiction. The same reasoning
applies to $D_{12}$, $D_{13}$, $D_{18}$, $D_{21}$, $D_{24}$,
$D_{26}$, and $D_{27}$.

\end{enumerate}

\noindent That leaves only the two diagrams $D_{17}$ and $D_{22}$
which clearly correspond to the two above mentioned patterns. Now we
see that for $D_{22}$  we must have the equalities $
p_1+p_4=p_1+p_6= p_1+p_7=p_6+p_7. $ This forces $p_1=p_4=p_6=p_7$.
In summary this pattern can only support the composition
$(u,0,0,u,0,u,u,0)$. The same reasoning yields that the diagram
$D_{17}$ can only support the composition $(0,v,v,0,v,0,0,v)$. It
follows that the general solution of $\CS_3$ must be of the form $
(a,b,c,d,d,c,b,a)+(u,v,v,u,v,u,u,v).
 $

Now recall that after the subtraction of a symmetric solution we are
left with an \emph{ asymmetric} solution. Thus to avoid over
counting we must impose the condition $u\, v=0$. This leaves only
three possibilities $u=v=0$, $u>0,v=0$ or $u=0,v>0$. Thus
\begin{align*}
 F_3(x) &=\sum_{a \ge 0}\sum_{ b \ge 0}\sum_{  c
\ge 0}\sum_{ d\ge 0} (x_1x_8)^a(x_2x_7)^b(x_3x_6)^c(x_4x_5)^d\Big(
1\sps \sum_{u\ge 1}(x_1x_4x_6x_7)^u \sps \sum_{v\ge
1}(x_2x_3x_5x_8)^v \Big) \cr & = {1 \over
(1-x_1x_8)(1-x_2x_7)(1-x_3x_6)(1-x_4x_5)  } \left(1\sps
{x_1x_4x_6x_7\over 1-  x_1x_4x_6x_7} \sps  {x_2x_3x_5x_8\over
1-x_2x_3x_5x_8}\right).\end{align*} which is only another way of
writing \eqref{4.8}. \sas
\def \oe  {\overline\eee}

\subsection{Algebraic decompositions and group actions}
It is easy to see that the decomposition of a solution into a sum of
a symmetric plus an asymmetric solution can be carried out for
general $k$. In fact, note that if   $0\le i\le 2^k-1$ has binary
digits $\eee_1\eee_2\cdots \eee_k$ then the binary digits of
$2^k-1-i$ are $\oe_1\oe_2\cdots \oe_k$ (with $\oe=1-\eee$). Thus we
see from \eqref{4.5} that in each equation $p_i$ and $p_{2^k+1-i}$
appear with opposite signs.  This shows that for each $k\ge 2$ the
system $\CS_k$ has $2^{k-1}$ symmetric solutions, which may be
symbolically represented by the monomials $x_ix_{i'}$ for
$i=1,\dots,2^{k-1}$, where we use (and will often use) $i'$ to
denote $2^k+1-i$ when $k$ is fixed.

Proceeding as we did for $\CS_2$ and $\CS_3$ we arrive at a unique
decomposition of each solution of $\CS_k$ into
$$
(p_1,p_2,\ldots ,p_{2^k})= (u_1,u_2,\ldots ,u_2,u_1)\sps
(q_1,q_2,\ldots ,q_{2^k})
$$
with the first summand symmetric and the second asymmetric, that is
$ u_i= u_{i'} $ and $q_i q_{i'}=0$ for $1\le i\le 2^{k-1}$, and
thereby obtain a factorization of $F_k(x)$ in the form
\begin{align}
 F_k(x)= \bigg(\prod_{i=1}^{2^{k-1}}{1\over 1-x_ix_{i'}
}\bigg)F_k^A(x) \label{4.17}
\end{align}
with $F_k^A(x)$ denoting the complete generating function of the
asymmetric solutions. \sas

This given it is tempting to try to apply, in the general case,  the same process we used for $k=3$ and obtain
the rational function $F_k^A(x)$ by selecting the patterns that do contain
the support of an asymmetric solution. Note that the total number of asymmetric
patterns to be examined is $2^{2^{k-1}}$ which is already
$256$ for $k=4$. For $k=5$ the number grows to $65,536$ and doing this by hand is out of the question.
Moreover, it is easy to see, by going through a few cases, that even for $k=4$
the geometry of the Venn Diagrams is so intricate that the only way that we can find out  if a
given pattern contains the support of a solution is to solve the corresponding reduced system.

\medskip
Nevertheless, using some inherent symmetries of the problem, the
complexity of the task can be substantially reduced to permit the
construction of $G_5(q)$ by computer. To describe how this was done
we need some notation. We will start with the complete generating
function of the system $\CS_k$ as given in Remark \ref{rem-1.3},
that is
$$
F_k(x_1,x_2,\ldots ,x_{2^k})= \prod_{i=1}^{2^k}{1\over 1-x_i A_i
}\ssp \bigg|_{a_1^0a_2^0\cdots a_k^0}, 
$$
where $A_i = \prod_{i=1}^k a_i^{1-2\eee_i}$, with
$\eee_1\eee_2\cdots \eee_k$ being the binary digits of $i-1$. Note
that since (as we previously observed)
 the binary digits of $2^k-1-i$ are  $\oe_1\oe_2\cdots \oe_k$, we have $
A_{i'}  = 1/A_i . $ It then follows that
$$
{1-x_ix_{i'}\over  (1-x_iA_i )(1-x_{i'}A_{i'} )}= \left({1\over
1-x_iA_i  }\sps {x_{i'}/A_i\over 1-x_{i'}/A_i  } \right).
$$
Thus combining the factors containing $A_i$ and $A_{i'}$
 we may rewrite \eqref{4.17} in the form
\begin{align}
 F_k(x_1,x_2,\ldots ,x_{2^k})= \prod_{i=1}^{2^{k-1}}{1\over
1-x_ix_{i'}}\prod_{i=1}^{2^{k-1}} \left({1\over  (1-x_iA_i )}\sps
{x_{i'}/A_i\over (1-x_{i'}/A_i )} \right)\ssp
\bigg|_{a_1^0a_2^0\cdots a_k^0}. \label{4.21}
\end{align}
Comparing with \eqref{4.17}  we derive that the complete generating
function of the asymmetric solutions is given by the following sum.
\begin{align}
F_k^A(x)=
 \sum_{S\con[1,2^{k-1}]}
F_S(x), \label{4.22}
\end{align}
where
\begin{align}
 F_S(x)= \bigg(\prod _{i\notin S}{1\over  (1-x_iA_i )}\bigg)
\times \bigg(\prod _{i\in S}{x_{i'}/A_i\over  (1-x_{i'}/A_i
)}\bigg)\ssp \bigg|_{a_1^0a_2^0\cdots a_k^0}. \label{4.23}
\end{align}
In this way we have described our decomposition algebraically. Using
notation as of \eqref{4.5}, we can see that $F_S(x)$ is none other
than the complete generating function of the reduced system
$$
\sum_{i\notin S} p_i V_i\ssp\sps\ssp \sum_{i\in S} p_{i'} V_{i'}= 0
$$
with the added condition that $ p_{i'}\ge 1 $ for all $i\in S.$

Note that for $k=3$ the summands in \eqref{4.22} correspond
precisely to the $16$ patterns in \eqref{4.16} with the added
condition that the ``$*$'' in position $i\ge 5$ should represent
$p_i\ge 1$ in the corresponding solution vector.  This extra
condition is precisely what is needed to eliminate  overcounting.

Perhaps all this is  best understood with an example. For instance
for $k=3$  the  patterns
$$
(*,0,0,*,0,*,*,0)\ess\ess\ess \hbox{and }\ess\ess\ess (0,*,*,0,*,0,0,*)
$$
were the only ones that supported an asymmetric solution represent
the two reduced systems
$$
\CS_{\{14\}}= \bigg\|\mymatrix{ &p_1+p_4-p_6-p_7=0\cr &
p_1-p_4+p_6-p_7=0\cr &p_1-p_4-p_6+p_7=0\cr } \bigsp 
\CS_{\{23\}}= \bigg\|\mymatrix{ &p_2+p_3-p_5-p_8=0\cr
&p_2-p_3+p_5-p_8=0\cr &-p_2+p_3+p_5-p_8=0\cr }
$$
and correspond to the following two summands of \eqref{4.22} for
$k=3$ \begin{align} F_{\{1,4 \} }(x)&= {1\over 1-x_1 a_1a_2 a_3}
{1\over 1-x_4 a_1/a_2a_3} {x_6 a_2/a_1a_3\over 1-x_6 a_2/a_1 a_3}
{x_7 a_3/a_1a_2\over 1-x_7 a_3/a_1 a_2}\ssp \bigg|_{a_1^0a_2^0a_3^0}
= {x_1x_4x_6x_7\over 1-x_1x_4x_6x_7} \label{4.24}\\
F_{\{2,3 \} }(x) &= {1\over 1-x_2 a_1a_2/a_3} {1\over 1-x_3
a_1a_3/a_2} {x_5 a_2a_3/a_1 \over 1-x_5 a_2a_3/a_1 } {x_8
/a_1a_2a_3\over 1-x_8 /a_1 a_2a_3}\ssp
\bigg|_{a_1^0a_2^0a_3^0}\!\!\! = {x_2x_3x_5x_8\over 1-x_2x_3x_5x_8}.
\label{4.25}
\end{align}

A close look at these two expressions should reveal the  key
ingredient that needs to be added to our algorithms that will permit
reaching  $k=5$ in the Hdd and Sdd problems. Indeed we see that
$F_{\{1,4 \} }(x)$ goes onto $ F_{\{2,3 \} }(x)$ if we act on the
vector $(x_1,x_2, \cdots ,x_8)$ by the permutation
\begin{align}
 \sig= \left(\begin{matrix} 1 & 2 & 3 & 4  & 5  & 6  & 7 & 8 \cr 3 & 4 &
1 & 2  & 7  & 8  & 5 & 6 \cr\end{matrix}\right) \label{4.26}
\end{align} and on the
triple $(a_1,a_2,a_3)$ by the operation $a_2\RA a_2^{-1}$. In fact,
$\sig$ is none other than an image of the map
$(\eee_1,\eee_2,\eee_3)\RA (\eee_1,\overline \eee_2,\eee_3)$ on the
binary digits of $0,1,\ldots,7$, as we can easily see when we
replace each $i$ in \eqref{4.26} by the binary digits of $i-1$
$$
\sig= \left(\begin{matrix} 000 & 001 & 010 & 011 & 100 & 101 & 110 &
111 \cr 010 & 011 & 000 & 001 & 110 & 111 & 100 & 101 \cr
\end{matrix}\right).
$$

\newcommand\mypmatrix[1]{\left(\begin{matrix}
#1
\end{matrix}\right)}

What goes on  is quite simple. Recall that solutions $p$ of our
system $\CS_k$ can also be viewed as assignments of weights to the
vertices of the $k$-hypercube giving all hyperfaces equal weight.
Then clearly
 any rotation or reflection of the hypercube will carry this assignment onto an assignment
with the same property. Thus the Hyperoctahedral group $\CB_k$ will act on all the constructs
we used to solve $\CS_k$.
\sas

To make precise the action of $\CB_k$ on $[1,2^k]$ we need some
conventions. \sas
\begin{enumerate}
\myitem{(1)} We will view the  elements of $\CB_k$  as pairs
$(\aa,\eta)$ with a permutation $\aa=(\aa_1,\aa_2,\ldots \aa_k)\in
S_k$ and a binary vector $\eta=(\eta_1,\eta_2,\ldots ,\eta_k)$.

\myitem {(2)} Next, for any binary vector
$\eee=(\eee_1,\eee_2,\ldots ,\eee_k)$ let us set
\begin{align}
 (\aa,\eta)\eee=(\eee_{\aa_1}+\eta_1,\eee_{\aa_2}+\eta_2,\cdots
,\eee_{\aa_k}+\eta_k) \label{4.27}
\end{align}
with ``{\ita mod 2} '' addition.

\myitem {(3)} This given,  to each element $g=(\aa,\eta)\in \CB_k$
there corresponds a permutation $\sig(g)$ by setting
$$
\sig(g)=\mypmatrix{
 1 &  2 &\cdots &  {2^k}\cr
\sig_1 & \sig_2 &\cdots & \sig_{2^k}\cr }. 
$$
where $\sig_i=j$ if and only if the  $k$-vector $\eee=(\eee_1,\eee_2
,\ldots, \eee_k)$ giving the binary digits of $i-1$ is sent by
\eqref{4.27} onto the  $k$-vector giving the binary digits of $j-1$.
In particular we will set
\begin{align}
 g(x_1,x_2,\ldots ,x_{2^k})= (x_{\sig_1},x_{\sig_2},\ldots
,x_{\sig_{2^k}}). \label{4.29}
\end{align}

\myitem {(4)} In the same vein we will make $\CB_k$ act on the
$k$-tuple $(a_1,a_2,\ldots ,a_k)$ by setting, again for
$g=(\aaa,\eta)$
\begin{align}
 g(a_1,a_2,\ldots ,a_k)=
(a_{\aa_1}^{1-2\eta_1},a_{\aa_2}^{1-2\eta_2},\ldots
,a_{\aa_k}^{1-2\eta_k}). \label{4.30}
\end{align}

\end{enumerate}

\noindent With these conventions we can easily derive  that $ g
x_iA_i = x_{\sig_i}A_{\sig_i}. $ Thus
$$
g \prod_{i=1}^{2^k}{1\over 1-x_i A_i }\ssp \bigg|_{a_1^0a_2^0\cdots
a_k^0}= \prod_{i=1}^{2^k}{1\over 1-x_{\sig_i} A_{\sig_i} }\ssp
\bigg|_{a_1^0a_2^0\cdots a_k^0} = \prod_{i=1}^{2^k}{1\over 1-x_i A_i
}\ssp \bigg|_{a_1^0a_2^0\cdots a_k^0}\ssp ,
$$
from which we again derive the $\CB_k$ invariance of the complete
generating function $ F_k(x_1,x_2,\ldots ,x_{2^k}).$

If we let $\CB_{k-1}$ not only act on the indices $1,2,\ldots
,2^{k-1}$, but also on $1',2',\ldots ,{2^{k-1}}'$ by
$\sigma_{i'}=\sigma_{i}'$. Then $\CB_{k-1}$ permutes the summands in
\eqref{4.22} as well as the factors in the product
$$
 \prod_{i=1}^{2^{k-1}}\!\!{1\over 1-x_ix_{i'}}.
$$ Note further that if we only want the $q$-series $G_k(q)$ we can
reduce \eqref{4.22} to
\begin{align}
 G_k^A(q)=
 \sum_{S\con[1,2^{k-1}]}
G_S(q), \label{4.31}
\end{align}
where $ G_S(q)=F_S(x)\big|_{x_i=q}$. But if for some $g\in\CB_{k-1}$
we have
$$
 F_{S_1}(x_{\sig_1},x_{\sig_2},\ldots ,x_{\sig_{2^k}})\ses
 F_{S_2}(x_{ 1},x_{ 2},\ldots ,x_{ {2^k}}),
$$
then replacing each $x_i$ by $q$ converts this to the equality $
G_{S_1}(q)=G_{S_2}(q). $ That means that we need only compute the
constant terms in \eqref{4.31} for orbit representatives, then
replace \eqref{4.31} by a sum over orbit representatives multiplied
by orbit sizes. More precisely we get
\begin{align}
 G_k^A(q)=
 \sum_{i}
m_iG_{S_i}(q), \label{4.32}
\end{align}
where the sum ranges over all orbits and $m_i$ denotes the
cardinality of the orbit of the representative $F_{S_i}(x)$. In the
computer implementation we obtain orbit representatives as well as
orbit sizes, by acting with $\CB_{k-1}$ on the monomials $ M_S=
\prod_{i\in S} x_i.$

Thus for $k=3$ we found that the $16$ summands in \eqref{4.22} break
up into $6$ orbits but only $2$ of them do contribute to $F^A_3$.
They corresponds to the monomials $1$ and $x_1x_4$ with respective
orbit sizes $1$ and $2$. The orbit representative that corresponds
to $1$ is simply the case $S=\phi$ in \eqref{4.23} and that
corresponds to $x_1x_4$ is given in \eqref{4.24}.

Thus from \eqref{4.24}, \eqref{4.32} and \eqref{4.21}  we derive
(again) that
$$
G_3(q)\ses {1\over (1-q^2)^4}\Big(1+2 {q^4\over 1-q^4}\Big) \ses
{1\over (1-q^2)^4} {1+q^4\over 1-q^4}.
$$
For $k=4$ we have $2^8=256$ summands in  \eqref{4.22}  with $22$
orbits but only $11$ of these orbits do contribute to $F^A_4$. The
number of denominator factors for each term is $8$ which is still a
reasonable number for the partial fraction algorithm. The formula
for $F_4(x)$ obtained this way can be typed within a page, but we
would like to introduce a nicer $F_4(x)$ using the full group
$\CB_k$ instead of $\CB_{k-1}$, as we will do  in the next
paragraph. For $k=5$ we have $2^{16}$ summands in \eqref{4.22} with
$402$ orbits but only $341$ orbits do contribute to $F^A_5$. The
number of denominator factors for each term is $16$ which is out of
reach for the partial fraction algorithm to obtain $F^A_5(x)$.
Nevertheless, in  this manner  we can still  produce $G_5(q)$ in
about 15 minutes.

\medskip
The decomposition in \eqref{4.31} is only $\CB_{k-1}$ invariant, and
it is natural from the geometry of the hypercube labelings, to ask
of a $\CB_{k}$ invariant decomposition. To obtain such a
decomposition of $F_k(x)$ we will pair off the factors containing
$A_i$ and $A_{i'}$ by means of the more symmetric identity
$$
{1-x_ix_{2^k+1-i}\over  (1-x_iA_i )(1-x_{2^k+1-i}A_{2^k+1-i} )}=
\left(1\sps {x_iA_i\over   1-x_iA_i  }\sps {x_{i'}A_{i'}\over
1-x_{i'}A_{i'}  } \right)
$$
and derive that
$$
F_k(x)= \sum_{S\cup T\con[1,2^{k-1}] }  F_{S,T}(x),
$$
where $S$ and $T$ are disjoint and
$$
F_{S,T}(x)= \bigg( \prod_{i\in S}^k{x_iA_i\over 1-x_iA_i}\bigg)
\bigg(\prod_{i\in T}{x_i'/A_i\over   1-x_i'/A_i }\bigg).
$$
Note that every pair $(S,T)$ should be  identified with  the set $S
\cup \{i' : i\in T\} \con [1,2^k]$ when applying the action of
$\CB_k$.

\begin{Example}
For $k=3$ we have $3^4=81$ summands with $9$ orbits but only $2$
orbits do contribute to $F^A_3$. The two orbits corresponds to the
monomials $1$ and $x_1x_4x_6x_7$ with respective orbit sizes $1$ and
$2$. The orbit representative that corresponds to $1$ is simply the
case $F_{\phi,\phi}=1|_{a_1^0a_2^0a_3^0a_4^0}=1$ and that
corresponds to $x_1x_4x_6x_7$ is
$$F_{\{1,4\},\{2,3\}}(x)= {x_1A_1 \over 1-x_1A_1}{x_4A_4 \over 1-x_4A_4}{x_6A_6 \over 1-x_6A_6} {x_7A_7 \over 1-x_7A_7}\Big|_{a_1^0a_2^0a_3^0a_4^0}
={x_1x_4x_6x_7 \over 1- x_1x_4x_6x_7} .  $$ Therefore, we can
reobtain $G_3(q)$ by using \eqref{4.32} as follows.
$$G_3(q)=\frac{1}{(1-q^2)^4}(1+\frac{2q^4}{1-q^4})=\frac{1+q^4}{(1-q^2)^4(1-q^4)}.$$
\end{Example}

\begin{Example}\label{exa-k4}
For $k=4$ we have $3^8=6561$ summands with $62$ orbits but only $10$
orbits do contribute to $F^A_4$. We obtain the following complete
generating functions for the $10$ orbit representatives:
$$
 (1)  \ess 1
$$
$$
 (24)  \ess  {  x_1    x_{15}    x_4
 x_{14}  \over 1-x_1x_{15}x_4x_{14}}
$$
$$
(16) \ess {  x_{16}    x_7  \left({ x_9 }\right)^{2} x_6    x_4
\over
 1-x_{16}x_7{x_9}^{2}x_6x_4}
$$
$$
 (96) \ess  { x_{15}
   x_3    x_7  \left({ x_{12}
 }\right)^{2}\left({  x_9  }\right)^{2}\left({
 x_6  }\right)^{3}\over
 \left({1-x_{12}x_7x_9x_6}\right)\left({1-x_{15}x_3x_{12}x_9{x_6}^{2}}\right)}
$$
$$
(96) \ess   {  x_{16}    x_{14}    x_5
 x_7  \left({  x_{11}
}\right)^{2}\left({  x_2  }\right)^{2}\over
\left({1-x_2x_7x_{11}x_{14}}\right)\left({1-x_{16}x_5x_2x_{11}}\right)}
$$
$$
(192)\ess   {  x_9    x_{10}    x_1  \left({ x_4 }\right)^{4}\left({
x_{15}  }\right)^{3}\left({ x_5 }\right)^{2}\left({ x_{14}
}\right)^{2}\over
 \left({1-x_1x_{15}x_4x_{14}}\right)\left({1-x_{15}
x_5x_{10}x_4}\right)\left({1-x_{15}x_5x_9{x_4}^{2}x_{14}}\right)}
$$
$$
(64) \ess  {  x_6    x_{16}    x_4  \left({ x_3 }\right)^{2}\left({
x_5  }\right)^{2}\left({ x_{15} }\right)^{2}\left({ x_{10}
}\right)^{3}\over
 \left({1-x_{15}x_5x_{10}x_4}\right)\left({1-x_{16}x_5x_3x_{10}}\right)
\left({1-x_{15}x_3x_{10}x_6}\right)}
$$
$$
(64) \ess  {  x_3    x_7    x_4  \left({ x_6 }\right)^{5}  x_1
\left({
 x_9  }\right)^{3}\left({  x_{12}
}\right)^{3}\left({  x_{15}  }\right)^{3}\over
 \left({1-x_1x_{15}x_{12}x_6}\right)\left({1-x_{12}x_7x_9x_6}\right)
\left({1-x_{15}x_9x_6x_4}\right)\left({1-x_{15}x_3x_{12}x_9{x_6}^{2}}\right)}
$$
$$
(32) \ess {\left({  x_{13}  }\right)^{3}\left({ x_{12} }\right)^{3}
x_1    x_3  x_2  x_6  x_7
  x_8
\left({1-x_1x_2x_3x_8{x_{12}}^{3}x_7{x_{13}}^{3}x_6}\right)\over
 \left({1-x_1x_8x_{12}x_{13}}\right)\left({1-x_2x_{12}x_7x_{13}}\right)
\left({1-x_3x_{12}x_{13}x_6}\right)\left({1-x_1{x_{12}}^{2}x_7x_{13}x_6}\right)
\left({1-x_2x_3x_8x_{12}{x_{13}}^{2}}\right)}
$$

$$
   {(8) \ess\ess x_4    x_5    x_3
 x_6    x_9   x_{10}
   x_{15}    x_{16}
\left({1-2\,x_{15}x_{16}x_5x_3x_{10}x_9x_6x_4+{x_{15}}^{2}{x_{16}}^{2}{x_5}^{2}
{x_3}^{2}{x_{10}}^{2}{x_9}^{2}{x_6}^{2}{x_4}^{2}}\right)\over
 \left({1-x_{16}x_3x_9x_6}\right)
\left({1-x_{16}x_5x_9x_4}\right)
\left({1-x_{15}x_9x_6x_4}\right)
\left({1-x_{15}x_5x_{10}x_4}\right)
\left({1-x_{16}x_5x_3x_{10}}\right)
\left({1-x_{15}x_3x_{10}x_6}\right)}
 $$
Here the numbers in parentheses give the respective orbit sizes.

Replacing all the $x_i$ by $q$ and summing as in \eqref{4.32}, we
obtain
$$
G_4(q)=
{1+q^2+21q^4+36q^6+74q^8+86q^{10}+74q^{12}+36q^{14}+21q^{16}+q^{18}+q^{20}
\over (1-q^2)^{7}(1-q^4)^4(1-q^6)}.
$$
We should mention that the partial fraction algorithm delivers this
rational function in less than a second by directly computing the
constant term in \eqref{I.12} for $k=4$. We computed the above
representatives because it contains more information and can be used
for an alternate path to $G_5(q)$.
\end{Example}

Computing the orbit representatives for $k=5$ requires the
construction of the $2^5\times 5!=3840$ elements of $\CB_5$ and
examining their action on  the $3^{16}=43046721$ symmetric supports.
This took a few hours on our computers. We found in this manner that
the $43046721$ summands in \eqref{4.31} break up into $ 15418 $
orbits and of these $6341$ contribute to the sum. Most of the orbits
have denominators of less than $16$ factors. It also took about $15$
minutes to persuade MAPLE to deliver $G_5(q)$ in the form displayed
in the introduction.

It turns out that the same orbit reduction idea can also be used to
compute $W_5(q)$, but much more complicated. Let us explain the
details in the next subsection.

\begin{Remark}
  \label{rem-4.1}
It is interesting to point out that  computing complete generating
functions for orbit representatives of summands in \eqref{4.22}
yielded as a byproduct orbit representatives of the extreme rays of
our Diophantine cone for $k=4$ and $k=5$. Note that for  $k=3$ the
representatives can be directly derived from our hand computation,
there are only two and the corresponding Venn Diagrams are

\hspace{2.25cm}\vspace{0.5cm}
\begin{picture}(100,80)
\put(0,0){\line(10,0){60}} \put(0,0){\line(0,10){60}}
\put(60,0){\line(0,10){60}} \put(0,60){\line(10,0){60}}
\put(30,35){\circle{25}}
\put(23.75,24.2){\circle{25}}\put(36.25,24.2){\circle{25}}
\put(29,38){\tiny{1}}\put(18,20){\tiny{1}}\put(40,20){\tiny{1}}
\put(120,0){\textrm{and}}

\put(195,0){ \put(0,0){\line(10,0){60}} \put(0,0){\line(0,10){60}}
\put(60,0){\line(0,10){60}} \put(0,60){\line(10,0){60}}
\put(30,35){\circle{25}}
\put(23.75,24.2){\circle{25}}\put(36.25,24.2){\circle{25}}
\put(29,26.5){\tiny{1}}\put(29,7){\tiny{1}} }
\end{picture}

\noindent Here the regions without numbers are empty. The number $1$
indicates that the region has only one element. For $k=4$ we found
that there are only three orbits, containing $24,8$ and $16$
elements respectively, the corresponding diagrams are depicted
below.

\hspace{-0.5cm} \vspace{0.25cm}
\begin{picture}(60,80)
\put(0,0){\line(10,0){60}} \put(0,0){\line(0,10){60}}
\put(60,0){\line(0,10){60}} \put(0,60){\line(10,0){60}}
\put(30,35){\circle{25}}
\put(23.75,24.2){\circle{25}}\put(36.25,24.2){\circle{25}}
\put(29,26.5){\tiny{1}}

\put(29,65){\tiny{\texttt{$\mathbf{A_1}$}}}
\put(29,52){\tiny{\texttt{$\mathbf{A_2}$}}}
\put(8,10){\tiny{\texttt{$\mathbf{A_4}$}}}
\put(48,10){\tiny{\texttt{$\mathbf{A_3}$}}}

\put(60,0){\put(0,0){\line(10,0){60}} \put(0,0){\line(0,10){60}}
\put(60,0){\line(0,10){60}} \put(0,60){\line(10,0){60}}
\put(30,35){\circle{25}}
\put(23.75,24.2){\circle{25}}\put(36.25,24.2){\circle{25}}
\put(29,26.5){\tiny{1}}

\put(27,65){\tiny{\texttt{$^c\mathbf{A_1}$}}}
\put(29,52){\tiny{\texttt{$\mathbf{A_2}$}}}
\put(8,10){\tiny{\texttt{$\mathbf{A_4}$}}}
\put(48,10){\tiny{\texttt{$\mathbf{A_3}$}}}}

\put(150,0){\put(0,0){\line(10,0){60}} \put(0,0){\line(0,10){60}}
\put(60,0){\line(0,10){60}} \put(0,60){\line(10,0){60}}
\put(30,35){\circle{25}}
\put(23.75,24.2){\circle{25}}\put(36.25,24.2){\circle{25}}
\put(29,26.5){\tiny{1}}\put(29,38){\tiny{1}}\put(29,7){\tiny{1}}

\put(29,65){\tiny{\texttt{$\mathbf{A_1}$}}}
\put(29,52){\tiny{\texttt{$\mathbf{A_2}$}}}
\put(8,10){\tiny{\texttt{$\mathbf{A_4}$}}}
\put(48,10){\tiny{\texttt{$\mathbf{A_3}$}}}

\put(60,0){\put(0,0){\line(10,0){60}} \put(0,0){\line(0,10){60}}
\put(60,0){\line(0,10){60}} \put(0,60){\line(10,0){60}}
\put(30,35){\circle{25}}
\put(23.75,24.2){\circle{25}}\put(36.25,24.2){\circle{25}}
\put(18,20){\tiny{1}}\put(40,20){\tiny{1}}

\put(27,65){\tiny{\texttt{$^c\mathbf{A_1}$}}}
\put(29,52){\tiny{\texttt{$\mathbf{A_2}$}}}
\put(8,10){\tiny{\texttt{$\mathbf{A_4}$}}}
\put(48,10){\tiny{\texttt{$\mathbf{A_3}$}}}}}

\put(300,0){\put(0,0){\line(10,0){60}} \put(0,0){\line(0,10){60}}
\put(60,0){\line(0,10){60}} \put(0,60){\line(10,0){60}}
\put(30,35){\circle{25}}
\put(23.75,24.2){\circle{25}}\put(36.25,24.2){\circle{25}}
\put(29,38){\tiny{1}}\put(18,20){\tiny{1}}\put(40,20){\tiny{1}}

\put(29,65){\tiny{\texttt{$\mathbf{A_1}$}}}
\put(29,52){\tiny{\texttt{$\mathbf{A_2}$}}}
\put(8,10){\tiny{\texttt{$\mathbf{A_4}$}}}
\put(48,10){\tiny{\texttt{$\mathbf{A_3}$}}}

\put(60,0){\put(0,0){\line(10,0){60}} \put(0,0){\line(0,10){60}}
\put(60,0){\line(0,10){60}} \put(0,60){\line(10,0){60}}
\put(30,35){\circle{25}}
\put(23.75,24.2){\circle{25}}\put(36.25,24.2){\circle{25}}
\put(29,26.5){\tiny{2}}\put(29,7){\tiny{1}}

\put(27,65){\tiny{\texttt{$^c\mathbf{A_1}$}}}
\put(29,52){\tiny{\texttt{$\mathbf{A_2}$}}}
\put(8,10){\tiny{\texttt{$\mathbf{A_4}$}}}
\put(48,10){\tiny{{$\mathbf{A_3}$}}}}}

\end{picture}

 Note, for $k=4$ each Venn diagram is
depicted as a pair of  Venn diagrams of $k=3$. The first member of the pair renders the Venn diagram
of $  A_1\cap A_2,A_1\cap A_3,A_1\cap A_4$  and the second member renders the Venn diagram
of $  \ ^cA_1\cap A_2, \ ^cA_1\cap A_3, \ ^cA_1\cap A_4$.
\sas

For $k=5$ we found that there are $2712$ extreme rays which break up
into $9$ orbits. We give in Figure \ref{fig-ext5} a set of
representatives depicted as assignments of weights to the vertices
of the $5$ dimensional hypercube. We imagine that the vertices of
this hypercube are indexed by the binary  digits  of  $0,1,2,\ldots
,31$ with $00000$ the vertex  at the origin and $11111$ giving the
coordinates of the opposite vertex. In Figure \ref{fig-ext5}  each
hypercube is represented by two rows of two cubes. The cubes in the
first row, from left to right, have the  vertices labeled with the
binary digits of 1 to 16 (minus 1) and the cubes in the second row
have the vertices labeled with the binary digits of 17 to 32 (minus
1). The vertices here have possible weights $0,1,2,3$ and,
correspondingly, are surrounded by $0,1,2,3$ concentric circles. The
integer on the top of each diagram gives the size of the
corresponding orbit.

\begin{figure}[hbt]
\begin{picture}(600,255)
\put(70,0){

\put(0,0){\line(10,0){20}}\put(20,0){\line(0,10){20}}
\put(0,20){\line(10,0){20}}\put(0,0){\line(0,10){20}}
\put(0,0){\line(1,1){10}}\put(20,0){\line(1,1){10}}
\put(0,20){\line(1,1){10}} \put(20,20){\line(1,1){10}}
\put(10,10){\line(10,0){20}}\put(30,10){\line(0,10){20}}
\put(10,30){\line(10,0){20}}\put(10,10){\line(0,10){20}}
\put(0,20){\circle{5}}\put(30,30){\circle{5}}

\put(32,70){\tiny{\textbf{960}}}

\put(35,0){\put(0,0){\line(10,0){20}}\put(20,0){\line(0,10){20}}
\put(0,20){\line(10,0){20}}\put(0,0){\line(0,10){20}}
\put(0,0){\line(1,1){10}}\put(20,0){\line(1,1){10}}
\put(0,20){\line(1,1){10}} \put(20,20){\line(1,1){10}}
\put(10,10){\line(10,0){20}}\put(30,10){\line(0,10){20}}
\put(10,30){\line(10,0){20}}\put(10,10){\line(0,10){20}}
\put(10,10){\circle{5}}\put(20,0){\circle{5}}

}

\put(0,35){\put(0,0){\line(10,0){20}}\put(20,0){\line(0,10){20}}
\put(0,20){\line(10,0){20}}\put(0,0){\line(0,10){20}}
\put(0,0){\line(1,1){10}}\put(20,0){\line(1,1){10}}
\put(0,20){\line(1,1){10}} \put(20,20){\line(1,1){10}}
\put(10,10){\line(10,0){20}}\put(30,10){\line(0,10){20}}
\put(10,30){\line(10,0){20}}\put(10,10){\line(0,10){20}}
\put(0,0){\circle{5}}\put(0,0){\circle{3}}

\put(35,0){\put(0,0){\line(10,0){20}}\put(20,0){\line(0,10){20}}
\put(0,20){\line(10,0){20}}\put(0,0){\line(0,10){20}}
\put(0,0){\line(1,1){10}}\put(20,0){\line(1,1){10}}
\put(0,20){\line(1,1){10}} \put(20,20){\line(1,1){10}}
\put(10,10){\line(10,0){20}}\put(30,10){\line(0,10){20}}
\put(10,30){\line(10,0){20}}\put(10,10){\line(0,10){20}}
\put(30,30){\circle{5}}\put(30,30){\circle{3}}
 }}

\put(90,0){\put(0,0){\line(10,0){20}}\put(20,0){\line(0,10){20}}
\put(0,20){\line(10,0){20}}\put(0,0){\line(0,10){20}}
\put(0,0){\line(1,1){10}}\put(20,0){\line(1,1){10}}
\put(0,20){\line(1,1){10}} \put(20,20){\line(1,1){10}}
\put(10,10){\line(10,0){20}}\put(30,10){\line(0,10){20}}
\put(10,30){\line(10,0){20}}\put(10,10){\line(0,10){20}}
\put(30,30){\circle{5}}

 \put(32,70){\tiny{\textbf{32}}}

\put(35,0){\put(0,0){\line(10,0){20}}\put(20,0){\line(0,10){20}}
\put(0,20){\line(10,0){20}}\put(0,0){\line(0,10){20}}
\put(0,0){\line(1,1){10}}\put(20,0){\line(1,1){10}}
\put(0,20){\line(1,1){10}} \put(20,20){\line(1,1){10}}
\put(10,10){\line(10,0){20}}\put(30,10){\line(0,10){20}}
\put(10,30){\line(10,0){20}}\put(10,10){\line(0,10){20}}
\put(10,30){\circle{5}}\put(20,20){\circle{5}}\put(30,10){\circle{5}}

}

\put(0,35){\put(0,0){\line(10,0){20}}\put(20,0){\line(0,10){20}}
\put(0,20){\line(10,0){20}}\put(0,0){\line(0,10){20}}
\put(0,0){\line(1,1){10}}\put(20,0){\line(1,1){10}}
\put(0,20){\line(1,1){10}} \put(20,20){\line(1,1){10}}
\put(10,10){\line(10,0){20}}\put(30,10){\line(0,10){20}}
\put(10,30){\line(10,0){20}}\put(10,10){\line(0,10){20}}
\put(0,0){\circle{5}}\put(0,0){\circle{3}}\put(0,0){\circle{7}}

\put(35,0){\put(0,0){\line(10,0){20}}\put(20,0){\line(0,10){20}}
\put(0,20){\line(10,0){20}}\put(0,0){\line(0,10){20}}
\put(0,0){\line(1,1){10}}\put(20,0){\line(1,1){10}}
\put(0,20){\line(1,1){10}} \put(20,20){\line(1,1){10}}
\put(10,10){\line(10,0){20}}\put(30,10){\line(0,10){20}}
\put(10,30){\line(10,0){20}}\put(10,10){\line(0,10){20}}
\put(30,30){\circle{5}} }}}

\put(180,0){\put(0,0){\line(10,0){20}}\put(20,0){\line(0,10){20}}
\put(0,20){\line(10,0){20}}\put(0,0){\line(0,10){20}}
\put(0,0){\line(1,1){10}}\put(20,0){\line(1,1){10}}
\put(0,20){\line(1,1){10}} \put(20,20){\line(1,1){10}}
\put(10,10){\line(10,0){20}}\put(30,10){\line(0,10){20}}
\put(10,30){\line(10,0){20}}\put(10,10){\line(0,10){20}}

\put(30,10){\circle{5}}

\put(32,70){\tiny{\textbf{320}}}

\put(35,0){\put(0,0){\line(10,0){20}}\put(20,0){\line(0,10){20}}
\put(0,20){\line(10,0){20}}\put(0,0){\line(0,10){20}}
\put(0,0){\line(1,1){10}}\put(20,0){\line(1,1){10}}
\put(0,20){\line(1,1){10}} \put(20,20){\line(1,1){10}}
\put(10,10){\line(10,0){20}}\put(30,10){\line(0,10){20}}
\put(10,30){\line(10,0){20}}\put(10,10){\line(0,10){20}}
\put(10,30){\circle{5}}\put(10,30){\circle{3}}\put(20,20){\circle{5}}\put(20,20){\circle{3}}
 }

\put(0,35){\put(0,0){\line(10,0){20}}\put(20,0){\line(0,10){20}}
\put(0,20){\line(10,0){20}}\put(0,0){\line(0,10){20}}
\put(0,0){\line(1,1){10}}\put(20,0){\line(1,1){10}}
\put(0,20){\line(1,1){10}} \put(20,20){\line(1,1){10}}
\put(10,10){\line(10,0){20}}\put(30,10){\line(0,10){20}}
\put(10,30){\line(10,0){20}}\put(10,10){\line(0,10){20}}
\put(0,0){\circle{5}}\put(0,0){\circle{3}}\put(0,0){\circle{7}}\put(30,30){\circle{5}}

\put(35,0){\put(0,0){\line(10,0){20}}\put(20,0){\line(0,10){20}}
\put(0,20){\line(10,0){20}}\put(0,0){\line(0,10){20}}
\put(0,0){\line(1,1){10}}\put(20,0){\line(1,1){10}}
\put(0,20){\line(1,1){10}} \put(20,20){\line(1,1){10}}
\put(10,10){\line(10,0){20}}\put(30,10){\line(0,10){20}}
\put(10,30){\line(10,0){20}}\put(10,10){\line(0,10){20}}
\put(30,10){\circle{5}}
 }}}

\put(0,90){\put(0,0){\line(10,0){20}}\put(20,0){\line(0,10){20}}
\put(0,20){\line(10,0){20}}\put(0,0){\line(0,10){20}}
\put(0,0){\line(1,1){10}}\put(20,0){\line(1,1){10}}
\put(0,20){\line(1,1){10}} \put(20,20){\line(1,1){10}}
\put(10,10){\line(10,0){20}}\put(30,10){\line(0,10){20}}
\put(10,30){\line(10,0){20}}\put(10,10){\line(0,10){20}}
\put(0,20){\circle{5}}

\put(32,70){\tiny{\textbf{384}}}

\put(35,0){\put(0,0){\line(10,0){20}}\put(20,0){\line(0,10){20}}
\put(0,20){\line(10,0){20}}\put(0,0){\line(0,10){20}}
\put(0,0){\line(1,1){10}}\put(20,0){\line(1,1){10}}
\put(0,20){\line(1,1){10}} \put(20,20){\line(1,1){10}}
\put(10,10){\line(10,0){20}}\put(30,10){\line(0,10){20}}
\put(10,30){\line(10,0){20}}\put(10,10){\line(0,10){20}}
\put(10,10){\circle{5}}\put(20,20){\circle{5}} }

\put(0,35){\put(0,0){\line(10,0){20}}\put(20,0){\line(0,10){20}}
\put(0,20){\line(10,0){20}}\put(0,0){\line(0,10){20}}
\put(0,0){\line(1,1){10}}\put(20,0){\line(1,1){10}}
\put(0,20){\line(1,1){10}} \put(20,20){\line(1,1){10}}
\put(10,10){\line(10,0){20}}\put(30,10){\line(0,10){20}}
\put(10,30){\line(10,0){20}}\put(10,10){\line(0,10){20}}
\put(0,0){\circle{5}}\put(30,10){\circle{5}}

\put(35,0){\put(0,0){\line(10,0){20}}\put(20,0){\line(0,10){20}}
\put(0,20){\line(10,0){20}}\put(0,0){\line(0,10){20}}
\put(0,0){\line(1,1){10}}\put(20,0){\line(1,1){10}}
\put(0,20){\line(1,1){10}} \put(20,20){\line(1,1){10}}
\put(10,10){\line(10,0){20}}\put(30,10){\line(0,10){20}}
\put(10,30){\line(10,0){20}}\put(10,10){\line(0,10){20}}
\put(30,30){\circle{5}}}}

\put(90,0){\put(0,0){\line(10,0){20}}\put(20,0){\line(0,10){20}}
\put(0,20){\line(10,0){20}}\put(0,0){\line(0,10){20}}
\put(0,0){\line(1,1){10}}\put(20,0){\line(1,1){10}}
\put(0,20){\line(1,1){10}} \put(20,20){\line(1,1){10}}
\put(10,10){\line(10,0){20}}\put(30,10){\line(0,10){20}}
\put(10,30){\line(10,0){20}}\put(10,10){\line(0,10){20}}
\put(30,10){\circle{5}}

\put(32,70){\tiny{\textbf{480}}}

\put(35,0){\put(0,0){\line(10,0){20}}\put(20,0){\line(0,10){20}}
\put(0,20){\line(10,0){20}}\put(0,0){\line(0,10){20}}
\put(0,0){\line(1,1){10}}\put(20,0){\line(1,1){10}}
\put(0,20){\line(1,1){10}} \put(20,20){\line(1,1){10}}
\put(10,10){\line(10,0){20}}\put(30,10){\line(0,10){20}}
\put(10,30){\line(10,0){20}}\put(10,10){\line(0,10){20}}
\put(20,0){\circle{5}}\put(10,30){\circle{5}} }

\put(0,35){\put(0,0){\line(10,0){20}}\put(20,0){\line(0,10){20}}
\put(0,20){\line(10,0){20}}\put(0,0){\line(0,10){20}}
\put(0,0){\line(1,1){10}}\put(20,0){\line(1,1){10}}
\put(0,20){\line(1,1){10}} \put(20,20){\line(1,1){10}}
\put(10,10){\line(10,0){20}}\put(30,10){\line(0,10){20}}
\put(10,30){\line(10,0){20}}\put(10,10){\line(0,10){20}}
\put(0,0){\circle{5}}\put(20,20){\circle{5}}

\put(35,0){\put(0,0){\line(10,0){20}}\put(20,0){\line(0,10){20}}
\put(0,20){\line(10,0){20}}\put(0,0){\line(0,10){20}}
\put(0,0){\line(1,1){10}}\put(20,0){\line(1,1){10}}
\put(0,20){\line(1,1){10}} \put(20,20){\line(1,1){10}}
\put(10,10){\line(10,0){20}}\put(30,10){\line(0,10){20}}
\put(10,30){\line(10,0){20}}\put(10,10){\line(0,10){20}}
\put(10,30){\circle{5}} }}}

\put(180,0){\put(0,0){\line(10,0){20}}\put(20,0){\line(0,10){20}}
\put(0,20){\line(10,0){20}}\put(0,0){\line(0,10){20}}
\put(0,0){\line(1,1){10}}\put(20,0){\line(1,1){10}}
\put(0,20){\line(1,1){10}} \put(20,20){\line(1,1){10}}
\put(10,10){\line(10,0){20}}\put(30,10){\line(0,10){20}}
\put(10,30){\line(10,0){20}}\put(10,10){\line(0,10){20}}
\put(30,30){\circle{5}}

\put(32,70){\tiny{\textbf{320}}}

\put(35,0){\put(0,0){\line(10,0){20}}\put(20,0){\line(0,10){20}}
\put(0,20){\line(10,0){20}}\put(0,0){\line(0,10){20}}
\put(0,0){\line(1,1){10}}\put(20,0){\line(1,1){10}}
\put(0,20){\line(1,1){10}} \put(20,20){\line(1,1){10}}
\put(10,10){\line(10,0){20}}\put(30,10){\line(0,10){20}}
\put(10,30){\line(10,0){20}}\put(10,10){\line(0,10){20}}
\put(0,20){\circle{5}}\put(30,10){\circle{5}}
 }

\put(0,35){\put(0,0){\line(10,0){20}}\put(20,0){\line(0,10){20}}
\put(0,20){\line(10,0){20}}\put(0,0){\line(0,10){20}}
\put(0,0){\line(1,1){10}}\put(20,0){\line(1,1){10}}
\put(0,20){\line(1,1){10}} \put(20,20){\line(1,1){10}}
\put(10,10){\line(10,0){20}}\put(30,10){\line(0,10){20}}
\put(10,30){\line(10,0){20}}\put(10,10){\line(0,10){20}}
\put(0,0){\circle{5}}\put(0,0){\circle{3}}

\put(35,0){\put(0,0){\line(10,0){20}}\put(20,0){\line(0,10){20}}
\put(0,20){\line(10,0){20}}\put(0,0){\line(0,10){20}}
\put(0,0){\line(1,1){10}}\put(20,0){\line(1,1){10}}
\put(0,20){\line(1,1){10}} \put(20,20){\line(1,1){10}}
\put(10,10){\line(10,0){20}}\put(30,10){\line(0,10){20}}
\put(10,30){\line(10,0){20}}\put(10,10){\line(0,10){20}}
\put(30,30){\circle{5}} }}}}

\put(0,180){\put(0,0){\line(10,0){20}}\put(20,0){\line(0,10){20}}
\put(0,20){\line(10,0){20}}\put(0,0){\line(0,10){20}}
\put(0,0){\line(1,1){10}}\put(20,0){\line(1,1){10}}
\put(0,20){\line(1,1){10}} \put(20,20){\line(1,1){10}}
\put(10,10){\line(10,0){20}}\put(30,10){\line(0,10){20}}
\put(10,30){\line(10,0){20}}\put(10,10){\line(0,10){20}}

\put(32,70){\tiny{\textbf{16}}}

\put(35,0){\put(0,0){\line(10,0){20}}\put(20,0){\line(0,10){20}}
\put(0,20){\line(10,0){20}}\put(0,0){\line(0,10){20}}
\put(0,0){\line(1,1){10}}\put(20,0){\line(1,1){10}}
\put(0,20){\line(1,1){10}} \put(20,20){\line(1,1){10}}
\put(10,10){\line(10,0){20}}\put(30,10){\line(0,10){20}}
\put(10,30){\line(10,0){20}}\put(10,10){\line(0,10){20}}
\put(20,20){\circle{5}} }

\put(0,35){\put(0,0){\line(10,0){20}}\put(20,0){\line(0,10){20}}
\put(0,20){\line(10,0){20}}\put(0,0){\line(0,10){20}}
\put(0,0){\line(1,1){10}}\put(20,0){\line(1,1){10}}
\put(0,20){\line(1,1){10}} \put(20,20){\line(1,1){10}}
\put(10,10){\line(10,0){20}}\put(30,10){\line(0,10){20}}
\put(10,30){\line(10,0){20}}\put(10,10){\line(0,10){20}}
\put(10,10){\circle{5}}

\put(35,0){\put(0,0){\line(10,0){20}}\put(20,0){\line(0,10){20}}
\put(0,20){\line(10,0){20}}\put(0,0){\line(0,10){20}}
\put(0,0){\line(1,1){10}}\put(20,0){\line(1,1){10}}
\put(0,20){\line(1,1){10}} \put(20,20){\line(1,1){10}}
\put(10,10){\line(10,0){20}}\put(30,10){\line(0,10){20}}
\put(10,30){\line(10,0){20}}\put(10,10){\line(0,10){20}}

}}

\put(90,0){\put(0,0){\line(10,0){20}}\put(20,0){\line(0,10){20}}
\put(0,20){\line(10,0){20}}\put(0,0){\line(0,10){20}}
\put(0,0){\line(1,1){10}}\put(20,0){\line(1,1){10}}
\put(0,20){\line(1,1){10}} \put(20,20){\line(1,1){10}}
\put(10,10){\line(10,0){20}}\put(30,10){\line(0,10){20}}
\put(10,30){\line(10,0){20}}\put(10,10){\line(0,10){20}}

\put(32,70){\tiny{\textbf{80}}}

\put(35,0){\put(0,0){\line(10,0){20}}\put(20,0){\line(0,10){20}}
\put(0,20){\line(10,0){20}}\put(0,0){\line(0,10){20}}
\put(0,0){\line(1,1){10}}\put(20,0){\line(1,1){10}}
\put(0,20){\line(1,1){10}} \put(20,20){\line(1,1){10}}
\put(10,10){\line(10,0){20}}\put(30,10){\line(0,10){20}}
\put(10,30){\line(10,0){20}}\put(10,10){\line(0,10){20}}
\put(20,20){\circle{5}}\put(10,30){\circle{5}} }

\put(0,35){\put(0,0){\line(10,0){20}}\put(20,0){\line(0,10){20}}
\put(0,20){\line(10,0){20}}\put(0,0){\line(0,10){20}}
\put(0,0){\line(1,1){10}}\put(20,0){\line(1,1){10}}
\put(0,20){\line(1,1){10}} \put(20,20){\line(1,1){10}}
\put(10,10){\line(10,0){20}}\put(30,10){\line(0,10){20}}
\put(10,30){\line(10,0){20}}\put(10,10){\line(0,10){20}}
\put(0,0){\circle{5}}\put(30,10){\circle{5}}

\put(35,0){\put(0,0){\line(10,0){20}}\put(20,0){\line(0,10){20}}
\put(0,20){\line(10,0){20}}\put(0,0){\line(0,10){20}}
\put(0,0){\line(1,1){10}}\put(20,0){\line(1,1){10}}
\put(0,20){\line(1,1){10}} \put(20,20){\line(1,1){10}}
\put(10,10){\line(10,0){20}}\put(30,10){\line(0,10){20}}
\put(10,30){\line(10,0){20}}\put(10,10){\line(0,10){20}}}}}

\put(180,0){\put(0,0){\line(10,0){20}}\put(20,0){\line(0,10){20}}
\put(0,20){\line(10,0){20}}\put(0,0){\line(0,10){20}}
\put(0,0){\line(1,1){10}}\put(20,0){\line(1,1){10}}
\put(0,20){\line(1,1){10}} \put(20,20){\line(1,1){10}}
\put(10,10){\line(10,0){20}}\put(30,10){\line(0,10){20}}
\put(10,30){\line(10,0){20}}\put(10,10){\line(0,10){20}}

\put(32,70){\tiny{\textbf{120}}}

\put(35,0){\put(0,0){\line(10,0){20}}\put(20,0){\line(0,10){20}}
\put(0,20){\line(10,0){20}}\put(0,0){\line(0,10){20}}
\put(0,0){\line(1,1){10}}\put(20,0){\line(1,1){10}}
\put(0,20){\line(1,1){10}} \put(20,20){\line(1,1){10}}
\put(10,10){\line(10,0){20}}\put(30,10){\line(0,10){20}}
\put(10,30){\line(10,0){20}}\put(10,10){\line(0,10){20}}

\put(10,10){\circle{5}}\put(20,20){\circle{5}}}

\put(0,35){\put(0,0){\line(10,0){20}}\put(20,0){\line(0,10){20}}
\put(0,20){\line(10,0){20}}\put(0,0){\line(0,10){20}}
\put(0,0){\line(1,1){10}}\put(20,0){\line(1,1){10}}
\put(0,20){\line(1,1){10}} \put(20,20){\line(1,1){10}}
\put(10,10){\line(10,0){20}}\put(30,10){\line(0,10){20}}
\put(10,30){\line(10,0){20}}\put(10,10){\line(0,10){20}}
\put(0,0){\circle{5}}\put(30,30){\circle{5}}

\put(35,0){\put(0,0){\line(10,0){20}}\put(20,0){\line(0,10){20}}
\put(0,20){\line(10,0){20}}\put(0,0){\line(0,10){20}}
\put(0,0){\line(1,1){10}}\put(20,0){\line(1,1){10}}
\put(0,20){\line(1,1){10}} \put(20,20){\line(1,1){10}}
\put(10,10){\line(10,0){20}}\put(30,10){\line(0,10){20}}
\put(10,30){\line(10,0){20}}\put(10,10){\line(0,10){20}}}}}}

}

\end{picture}
\caption{\label{fig-ext5}Representatives of extreme rays for $k=5$.}
\end{figure}
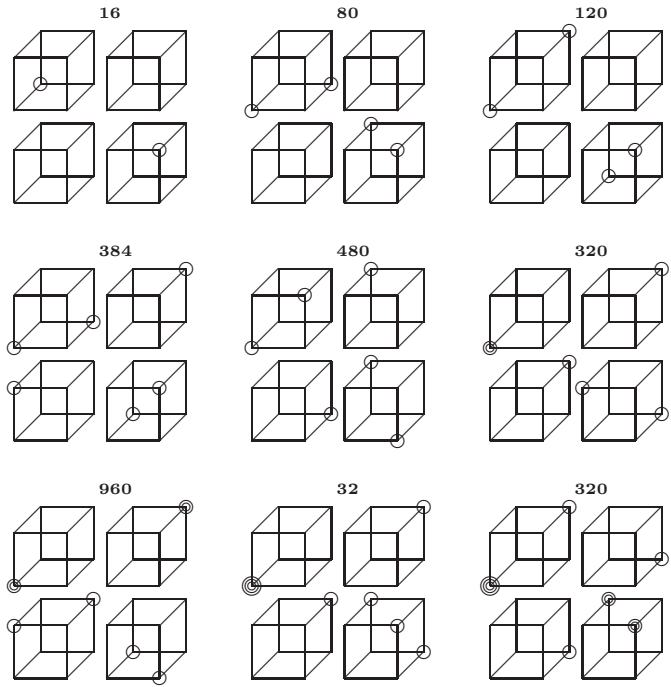

Each of the corresponding solutions of our system $\CS_5$ is \emph{
minimal}, that is, it cannot be decomposed into a non-trivial sum of
solutions. But we found that there are also $480$ minimal solutions
that do not come from extreme rays. The latter break up into two
orbits, with representatives depicted in Figure \ref{fig-min5}.

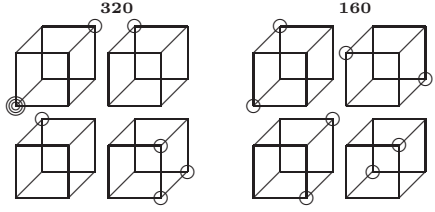
\begin{figure}[hbt]
\hspace{-9cm} \mbox{\begin{picture}(100,90)

\put(100,0){ \put(0,0){\line(10,0){20}}\put(20,0){\line(0,10){20}}
\put(0,20){\line(10,0){20}}\put(0,0){\line(0,10){20}}
\put(0,0){\line(1,1){10}}\put(20,0){\line(1,1){10}}
\put(0,20){\line(1,1){10}} \put(20,20){\line(1,1){10}}
\put(10,10){\line(10,0){20}}\put(30,10){\line(0,10){20}}
\put(10,30){\line(10,0){20}}\put(10,10){\line(0,10){20}}
\put(10,30){\circle{5}}

\put(32,70){\tiny{\textbf{320}}}

\put(35,0){\put(0,0){\line(10,0){20}}\put(20,0){\line(0,10){20}}
\put(0,20){\line(10,0){20}}\put(0,0){\line(0,10){20}}
\put(0,0){\line(1,1){10}}\put(20,0){\line(1,1){10}}
\put(0,20){\line(1,1){10}} \put(20,20){\line(1,1){10}}
\put(10,10){\line(10,0){20}}\put(30,10){\line(0,10){20}}
\put(10,30){\line(10,0){20}}\put(10,10){\line(0,10){20}}
\put(20,0){\circle{5}}\put(20,20){\circle{5}}\put(30,10){\circle{5}}

}

\put(0,35){\put(0,0){\line(10,0){20}}\put(20,0){\line(0,10){20}}
\put(0,20){\line(10,0){20}}\put(0,0){\line(0,10){20}}
\put(0,0){\line(1,1){10}}\put(20,0){\line(1,1){10}}
\put(0,20){\line(1,1){10}} \put(20,20){\line(1,1){10}}
\put(10,10){\line(10,0){20}}\put(30,10){\line(0,10){20}}
\put(10,30){\line(10,0){20}}\put(10,10){\line(0,10){20}}
\put(0,0){\circle{5}}\put(0,0){\circle{3}}\put(0,0){\circle{7}}
\put(30,30){\circle{5}}

\put(35,0){\put(0,0){\line(10,0){20}}\put(20,0){\line(0,10){20}}
\put(0,20){\line(10,0){20}}\put(0,0){\line(0,10){20}}
\put(0,0){\line(1,1){10}}\put(20,0){\line(1,1){10}}
\put(0,20){\line(1,1){10}} \put(20,20){\line(1,1){10}}
\put(10,10){\line(10,0){20}}\put(30,10){\line(0,10){20}}
\put(10,30){\line(10,0){20}}\put(10,10){\line(0,10){20}}
\put(10,30){\circle{5}}
 }}

 \put(90,0){\put(0,0){\line(10,0){20}}\put(20,0){\line(0,10){20}}
\put(0,20){\line(10,0){20}}\put(0,0){\line(0,10){20}}
\put(0,0){\line(1,1){10}}\put(20,0){\line(1,1){10}}
\put(0,20){\line(1,1){10}} \put(20,20){\line(1,1){10}}
\put(10,10){\line(10,0){20}}\put(30,10){\line(0,10){20}}
\put(10,30){\line(10,0){20}}\put(10,10){\line(0,10){20}}
\put(20,0){\circle{5}}\put(30,30){\circle{5}}

\put(32,70){\tiny{\textbf{160}}}

\put(35,0){\put(0,0){\line(10,0){20}}\put(20,0){\line(0,10){20}}
\put(0,20){\line(10,0){20}}\put(0,0){\line(0,10){20}}
\put(0,0){\line(1,1){10}}\put(20,0){\line(1,1){10}}
\put(0,20){\line(1,1){10}} \put(20,20){\line(1,1){10}}
\put(10,10){\line(10,0){20}}\put(30,10){\line(0,10){20}}
\put(10,30){\line(10,0){20}}\put(10,10){\line(0,10){20}}
\put(10,10){\circle{5}}\put(20,20){\circle{5}}

}

\put(0,35){\put(0,0){\line(10,0){20}}\put(20,0){\line(0,10){20}}
\put(0,20){\line(10,0){20}}\put(0,0){\line(0,10){20}}
\put(0,0){\line(1,1){10}}\put(20,0){\line(1,1){10}}
\put(0,20){\line(1,1){10}} \put(20,20){\line(1,1){10}}
\put(10,10){\line(10,0){20}}\put(30,10){\line(0,10){20}}
\put(10,30){\line(10,0){20}}\put(10,10){\line(0,10){20}}
\put(0,0){\circle{5}}\put(10,30){\circle{5}}

\put(35,0){\put(0,0){\line(10,0){20}}\put(20,0){\line(0,10){20}}
\put(0,20){\line(10,0){20}}\put(0,0){\line(0,10){20}}
\put(0,0){\line(1,1){10}}\put(20,0){\line(1,1){10}}
\put(0,20){\line(1,1){10}} \put(20,20){\line(1,1){10}}
\put(10,10){\line(10,0){20}}\put(30,10){\line(0,10){20}}
\put(10,30){\line(10,0){20}}\put(10,10){\line(0,10){20}}
\put(0,20){\circle{5}}\put(30,10){\circle{5}}
 }}}}
\end{picture}
}

\caption{Representatives for minimal solutions but not extreme
rays.\label{fig-min5}}
\end{figure}
\end{Remark}

\subsection{Our fastest way for $G_5(q)$ and $W_5(q)$}
With the notations in the previous subsection and Section
\ref{sec-3} handy, we can describe our best way to obtain $G_5(q)$
and $W_5(q)$.

Let us explain the idea for $k=5$. In Example \ref{exa-k4} we have
obtained for $F_4^A(x)$ 10 orbit representatives with corresponding
orbit sizes. Denote them by $R_i(x)$ the representatives and $m_i$
the orbit sizes for $i=1,\dots, 10$. From this we can give explicit
formula of $F_4^A(x)$ and hence of $F_4(x)$ with the help of $\CB_4$
action as follows.
\begin{align}\label{e-F4}
F_4(x)=\frac{F_4^A(x)}{\prod_{i=1}^8(1-x_ix_{17-i})}=
\sum_{i=1}^{10} \frac{m_i}{|\CB_4|} \sum_{g\in \CB_4} g
\frac{R_i}{\prod_{i=1}^8(1-x_ix_{17-i})}
\end{align}
Applying Algorithm \ref{algo-Fk} $\mathbf{a_5), b_5)}$ to
\eqref{e-F4}, we can obtain $F_5(x)$ by multilinearity.
\begin{align}
\label{e-F5} F_5(x)= \sum_{i=1}^{10} \frac{m_i}{|\CB_4|} \sum_{g\in
\CB_4} \left(g \delta_{1,17}\cdots \delta_{16,32}
\frac{R_i}{\prod_{i=1}^8(1-x_ix_{17-i})}\right)\bigg|^{j=1,2,\dots,16}_{x_j=x_ja,x_{16+j}=x_{16+j}/a}\Bigg|_{a^0}
,
\end{align}
where we have used the straightforwardly checked fact: for any
rational function $R(x_1,\dots,x_{16})$ and $g\in \CB_k$, it holds
that
$$\delta_{1,17}\cdots \delta_{16,32} \; g R(x_1,\dots,x_{16})= g \delta_{1,17}\cdots \delta_{16,32}  R(x_1,\dots,x_{16}),$$
where $g$ is extended to permute also indices $16+j$ by
$g(16+j)=16+g(j)$ for $j=1,\dots,16$.

Substituting $x_j=q$ for all $j$ into \eqref{e-F5} gives
\begin{align}
\label{e-divgroup}
 G_5(q)= \sum_{i=1}^{10} m_i \left(\delta_{1,17}\cdots
\delta_{16,32}
\frac{R_i}{\prod_{i=1}^8(1-x_ix_{17-i})}\right)\bigg|^{j=1,2,\dots,16}_{x_j=x_ja,x_{16+j}=x_{16+j}/a}\
\Bigg|_{a^0} .
\end{align} That is to say, we only need representatives of $F_{k-1}(x)$
together with orbit sizes to compute $F_k(x)$, and this clearly
extends for general $k$. Using \eqref{e-divgroup}, we can persuade
Maple to deliver $G_5(q)$ as in \eqref{I.14} in about 12 minutes.

\medskip
The orbit reduction idea for $G_5(q)$ works in a similar way for
$W_5(q)$. In fact, we can carry out almost verbatim the same steps
that yielded the orbit decomposition of the complete generating
function $F_k(x_1,x_2,\ldots ,x_{2^k})$ to obtain the complete
generating function  $\widetilde{W}_k(x_1,x_2,\ldots ,x_{2^k})$ as
we shall define.
 Recall that the $W_k(x)$ was originally defined in \eqref{1.18} as the constant
term
\begin{align}
 W_k(x_1,x_2,\ldots ,x_{2^k})\ses \prod_{j=1}^k(1-a_j^2)
\prod_{i=1}^{2^k}{1\over 1-x_iA_i}\bigg|_{a_1^0 a_2^0\cdots a_k^0}.
\label{4.35}
\end{align}
To carry out its decomposition we need only observe that if we let
\begin{align}
 \widetilde{W}_k(x_1,x_2,\ldots ,x_{2^k})\ses {1\over
2^k}\prod_{j=1}^k(1-a_j^2)(1-a_j^{-2}) \prod_{i=1}^{2^k}{1\over
1-x_iA_i}\bigg|_{a_1^0 a_2^0\cdots a_k^0}, \label{4.36}
\end{align}
then
$$
W_k(q)=\widetilde{W}_k(q).
$$
 The reason for this is that when all the $x_i$ are
replaced by $q$, we can easily show that
 the constant term
in \eqref{4.35} is not affected if we replace any $a_i$ by
$a_i^{-1}$. Thus if we average out the right hand side of
\eqref{4.35} over all these interchanges the result will be simply
the right hand side of \eqref{4.36} due to the simple relation
$$
  1-{a_i^2+a_i^{-2}\over 2 }\ses {1 \over 2 } (1-a_i^2)(1-a_i^{-2}).
$$

Now \eqref{4.36} brings to evidence that $\widetilde{W}_k(x)$ is
$\CB_k$ invariant while $W_k(x)$ is not. Symmetrizing $W_k(x)$ gives
$\widetilde{W}_k(x)$. We can obtain either a $\CB_{k-1}$ invariant
decomposition or a $\CB_k$ invariant decomposition of
$\widetilde{W}_k(x)$ just as for $F_k(x)$.

The orbit reduction can also be used to considerably speed up  steps
$\mathbf{a_k)}$  and $\mathbf{ b_k')}$ in Algorithm \ref{algo-Wk} of
the divided difference. The idea is similar as for the computation
of $G_5(q)$, but is much harder to be carried out.

To be clearer, we note that in step $ \mathbf{b_k')}$ we do not need
the complete generating function $W_{k-1}(x)$. One way is to replace
it by the more symmetric $\widetilde{W}_{k-1}(x)$. We have
$$
\widetilde{W}_{k-1}(x)=
\frac{1}{\prod_{i=1}^{2^{k-2}}(1-x_ix_{2^{k-1}+1-i})}\sum_{S\cup
T\con [1,,2^{k-2}]}\widetilde{W}_{S,T}(x),
$$
where $S$ and $T$ are disjoint as before. We only need to find orbit
representatives
$$
\widetilde{W}_{S_1,T_1}(x)\scs \widetilde{W}_{S_2,T_2}(x)\scs \ldots
,\widetilde{W}_{S_N,T_N}(x)
$$
with respective multiplicities $m_1,m_2,\ldots ,m_N$, since from
them we can rebuilt $\widetilde{W}_{k-1}(x)$, just as in
\eqref{e-F4}. Then in step $\mathbf{a_k)}$ we can replace
$\widetilde{W}_{k-1}(x)$ by the sum
$$
\widetilde{W}_{k-1}'(x)=
\frac{1}{\prod_{i=1}^{2^{k-2}}(1-x_ix_{2^{k-1}+1-i})}\sum_{i=1}^N
m_i \widetilde{W}_{S_i,T_i}(x)
$$
and, with a similar reasoning as for $G_k(q)$, obtain
\begin{align}
\label{e-Wk} W_k(q)= \sum_{i=1}^N m_i \Big( \dd_{1,1+2^{k-1}}\cdots
\dd_{2^{k-1}, 2^{k}}
\widetilde{W}_{S_i,T_i}(x)\Big)\Big|^{j=1,\dots,2^{k-1}-1}_{x_j=qa,x_{j+2^{k-1}=q/a}}
(1-a^2)\Bigg|_{a^0} .
\end{align}

When working with $W_5(q)$, we need an analogue of the collection of
orbit representatives together with orbit sizes as in Example
\ref{exa-k4}. Although Maple gives such a collection, we find it too
complicated to be handled by Maple when using \eqref{e-Wk}.

We find a way to avoid this problem. The idea is that in a formula
like \eqref{e-F4}, the $R_i$ need not be chosen to have
combinatorial meanings. This is best illustrated by the $k=3$ case.
We can clearly see the advantage of orbit reduction in producing a
compressed version of $\widetilde{W}_k(x)$. For $k=3$, the $\CB_3$
decomposition will give 9 orbits with only 7 of them contributing to
$\widetilde{W}_3(x)$. We thus get
$$\widetilde{W}^A_3(x)={1\over |\CB_3|} \sum_{g\in \CB_3} g \left(9\ monomials +{27 monomials \over 1-x_1x_4x_6x_7}\right)
.$$ The actual  formula is a little complicated and its
combinatorial meaning is not significant, but it is good enough for
us to use the divided
  difference algorithm  to compute $W_4(q)$. From this, by symmetrizing
and re-choosing   representatives, we obtain a simpler
representative.  Namely we end up obtaining that
$$\widetilde{W}^A_3(x)={1\over |\CB_3|} \sum_{g\in \CB_3} g
\left(-1+3\,x_{{2}}x_{{6}}-x_{{1}}x_{{2}}x_{{6}}x_{{4}} +
{2-6\,x_{{1}}x_{{7}}-{x_{{1}}}^{2}+6\,x_{{1}}{x_{{4}}}^{2}x_{{7}}-{x_{{1}}}^{2}{x_{{4}}}^{2}{x_{{7}}}^{2}\over
\left({1-x_{{1}}x_{{6}}x_{{4}}x_{{7}}}\right)}\right) 
,$$ which can also be used in our divided difference algorithm.
Originally we hoped that this formula would enable us to compute
$W_4(q)$ entirely by hand, but we were unable to do so.

For $k=4$, directly using the $\CB_4$ decomposition gives us $62$
orbits with $27$ of them contributing to $\widetilde{W}_4(x)$. The
representatives obtained this way are too complex for further
computation since several of them have thousands of monomials in
their numerators. The similar idea of symmetrizing and re-choosing
applies to give us $10$ reasonably simple representatives for
$\widetilde{W}_4(x)$, but typesetting them will take several pages.
Nevertheless we are able to use them in the divided difference
algorithm.

Having noticed that  for $k=2,3,4$ the divided difference algorithm
reduced the computation of $W_k(q)$ to a rather simple constant term
evaluation, we tried to see what it gave for $k=5$. Adding the
contributions of these $10$ representatives, before taking the
constant term, yielded a rational function of the form
\begin{multline*}
 {1 \over \left({1-{q}^{2}}\right)\left({1-{q}^{4}}\right)^{4}
\left({1-{q}^{6}}\right)\left(1-{{q}^{2}\over
{a}^{2}}\right)\left({1-{a}^{2}{q}^{2}}\right) \left(1-{{q}^{4}\over
{a}^{2}}\right)^{3}\left({1-{a}^{2}{q}^{4}}\right)^{3}}\\
  \times {357\ monomials \over
\left(1-{{q}^{4}\over
{a}^{4}}\right)^{2}\left({1-{a}^{4}{q}^{4}}\right)^{2}
\left(1-{{q}^{6}\over {a}^{2}}\right)\left({1-{a}^{2}{q}^{6}}\right)
\left(1-{{q}^{6}\over {a}^{4}}\right)\left({1-{a}^{4}{q}^{6}}\right)
\left(1-{{q}^{6}\over {a}^{6}}\right)
\left({1-{a}^{6}{q}^{6}}\right)  }.
\end{multline*}
It turns out that this is actually a rational function in $q^2$ and
$a^2$. Replacing $q$ by $q^{1/2}$ and $a$ by $a^{1/2}$ and then
taking constant term in $a$, we can obtain $W_5(q^{1/2})$. Using
this approach Maple can deliver $W_5(q)$ in only about 5 minutes in
total which is the shortest time we have been able to compute this
series.


\subsection{Our first algorithm to obtain $G_5(q)$ and $W_5(q)$}
Before closing it will be worthwhile to include a description of the
first algorithm that was used to obtain $G_5(q)$ and  $W_5(q)$ since
it contains another trick that clearly shows the flexibility
afforded by the partial fraction algorithm in the computation of
constant terms.

In this approach we begin by replacing our system  $\CS_k$ by a
system $\CS_k'$ which has the same cone of solutions. To describe
the new system we will use the $k$-tuple of sets model. The idea is
that originally we got $\CS_k$ by equating the cardinality of each
set to the cardinality of its complement obtaining
$$
\CS_k\ses \left \| \mymatrix{ |A_1|=|\ ^cA_1|\cr |A_2|=|\ ^cA_2|\cr
\cdots \cr |A_k|=|\ ^cA_k|\cr } \right. .
$$
Now it is quite clear that this is equivalent to set
\begin{align*}
 \CS_k'\ses \left \| \mymatrix{ |A_1|=d\cr |A_2|=d\cr \cdots \cr
|A_k|=d\cr |\ ^cA_1|=d} \right. .
\end{align*}
For instance, using the binary digit indexing of the variables, for
$k=3$ this results in the following system of 4 equations in 9
unknowns
\begin{align*}
\begin{array}{ccccccccccccccc} p_{000}&+& p_{001} &+& p_{010}  &+ & p_{011}    & &   & &  &
& & & \ess\ess\ess\ess\ssp \sms d\ses 0 \cr p_{000}&+& p_{001} &+& &
&      & +& p_{100}&+& p_{101}& & & & \ess\ess\ess\ess\ssp \sms
d\ses 0\cr
 p_{000}& &   &+& p_{010}  &  &      &+& p_{100}& &
&+& p_{110}  & & \ess\ess\ess\ess\ssp \sms d\ses 0\cr
  & &   & &   &  &     & & p_{100}&+& p_{101}&+& p_{110}&+& p_{111}   \ssp \sms d\ses 0
\end{array} \ . 
\end{align*}
This given, our rational function $G_3(q)=G_3(q,1)$ may be also
obtained by taking the following constant term
\begin{align}
\displaystyle{G_3(q,t)\ses \displaystyle{ 1\over1-q a_1a_2a_3}
{1\over1-q a_1a_2 }
 {1\over1-q a_1 a_3}
{1\over1-q a_1 } \rule{7cm}{0pt} \atop
\rule{3cm}{0pt}\ess\displaystyle {1\over1-q a_2a_3a_4} {1\over1-q
a_2a_4 } {1\over1-q a_3 a_4} {1\over1-q a_4 } {1\over
1-t/a_1a_2a_3a_4}\ssp \bigg|_{a_1^0a_2^0a_3^0a_4^0}. }\label{4.39}
\end{align}
Here we choose the order $q<t<a_1<a_2<\cdots$ and we can not set
$t=1$ as this moment yet.

Now it turns out to be expedient to start by eliminating $a_4$. This
can simply be done by omitting the factor $1/(1-t/a_1a_2a_3a_4)$ and
making the substitution   $a_4\,\RA\,  t/a_1a_2a_3$, obtaining
$$
\displaystyle{ G_3(q,t)\ses \displaystyle{1\over1-q a_1a_2a_3}
 {1\over1-q a_1a_2 }
{1\over1-q a_1 a_3} {1\over1-q a_1 }\rule{7cm}{0pt} \atop
\rule{3cm}{0pt} \displaystyle {1\over1-q t /a_1} {1\over1-qt /a_1a_3
} {1\over1-q t /a_1a_2} {1\over1-qt /a_1a_2a_3}\ssp
\bigg|_{a_1^0a_2^0a_3^0 }. } 
$$
Setting $t=1$ is valid here. Grouping terms containing the same
subset of the variables $a_1,a_2,a_3$ gives
\begin{align}
\qquad G_3(q)= {1\over1-q a_1 }&{1\over1-q  /a_1} \cr &\ess\ess
 {1\over1-q a_1a_2 }
{1\over1-q  /a_1a_2} \rule{9cm}{0pt} \label{4.41}\\
&\ess\ess\ess\ess\ess\ess\ess\ess\ess\ess\ess\ess {1\over1-q a_1
a_3}{1\over1-q /a_1a_3 } \cr
&\ess\ess\ess\ess\ess\ess\ess\ess\ess\ess\ess\ess\ess\ess\ess\ess\ess\ess\ess\ess\ess\ess\ess
 {1\over1-q a_1a_2a_3}{1\over1-q   /a_1a_2a_3}\ssp \bigg|_{a_1^0a_2^0a_3^0
 }.\nonumber
\end{align}
Likewise, we can easily see that the general form of \eqref{4.39} is
$$
G_k(q,t)\ses \bigg(\prod_{S\con[2,k]}{1\over 1- qa_1A(S) }\bigg)
\bigg(\prod_{S\con[2,k]}{1\over 1- qA(S)  a_{k+1} }\bigg) {1\over
1-t/a_1a_2\cdots a_k a_{k+1}}\ssp \bigg|_{a_1^0a_2^0\cdots
a_k^0a_{k+1}^0 }
$$
with
$$
A(S)\ses \prod_{i \in S}a_i.
$$
Removing the last factor and  setting $a_{k+1}=t/a_1a_2\cdots a_k$
gives
$$
G_k(q,t)\ses \bigg(\prod_{S\con[2,k]}{1\over 1- qa_1A(S) }\bigg)
\bigg(\prod_{S\con[2,k]}{1\over 1- qtA(S)/a_1a_2\cdots a_k }\bigg)
\bigg|_{a_1^0a_2^0\cdots a_k^0  }
$$
and by setting $t=1$ this can be rewritten as
$$
G_k(q)\ses \bigg(\prod_{S\con[2,k]}{1\over 1- qa_1A(S) }
 {1\over 1- q/a_1A(S)   }\bigg)
\bigg|_{a_1^0a_2^0\cdots a_k^0  }.
$$
Now comes the next trick: grouping terms according as $A(S)$
contains $a_2$ or not. This gives
\begin{align}
G_k(q)= \Bigg[\prod_{S\con[3,k]}{1\over 1- qa_1A(S) }
 {1\over 1- \frac{q}{a_1A(S)}   }\Bigg]
\Bigg[\prod_{S\con[3,k]}{1\over 1- qa_1a_2A(S) }
 {1\over 1- \frac{q}{a_1a_2A(S)}   }\Bigg]
\Bigg|_{a_1^0a_2^0\cdots a_k^0  }. \label{4.42}
\end{align}
To appreciate the significance of this step  let us see what this
gives for $k=3$. Grouping terms  in \eqref{4.41} as was done in
\eqref{4.42} gives
\begin{multline}
G_3(q) = {1\over1-q a_1 }{1\over1-q  /a_1} {1\over1-q a_1
a_3}{1\over1-q /a_1a_3  } \\
{1\over1-q a_1a_2 }{1\over1-q /a_1a_2}
{1\over 1-q a_1a_2a_3} {1\over 1-q  /a_1a_2a_3}\ssp
\bigg|_{a_1^0a_2^0a_3^0 }. \label{4.43}
\end{multline}

Let us now see what the partial fraction algorithm gives if we first
eliminate $a_2$. This entails computing the constant term
$$
Q\ses {1\over1-q a_1a_2 }{1\over1-q  /a_1a_2} {1\over 1-q a_1a_2a_3}
{1\over 1-q   /a_1a_2a_3}\ssp \bigg|_{ a_2^0  }.
$$
Using the terminology of \cite{2} we note that the first and third
factors are  contributing and the other two are dually contributing.
Thus,
\begin{align}
 Q\ses {A_1\over 1-qa_1a_2}\sps {A_3\over 1-qa_1a_2a_3}\ssp
\bigg|_{ a_2^0  }\ses A_1+A_3 \label{4.44}
\end{align}
with
\begin{align*}
A_1 &= {a_1^2a_2^2  a_3 \over   (a_1a_2-q)    (1-q a_1a_2a_3)
   (a_1a_2a_3-q)   }\ssp \bigg|_{ a_2=1/qa_1 }
= {1\over
   (1-q^2)(1-a_3)(1-q^2/a_3)} 
   \\
A_3 &= {a_1^2a_2^2  a_3 \over (1-q a_1a_2) (a_1a_2-q)
   (a_1a_2a_3-q)   }\ssp  \bigg|_{ a_2=1/qa_1a_3 }
= { a_3\over (a_3-1) (1-q^2a_3)(1-q^2)}.
\end{align*}
Using \eqref{4.44} in \eqref{4.43} gives
\begin{align}
G_3(q) &={1\over1-q a_1 }{1\over1-q /a_1} {1\over1-q a_1 a_3}
{1\over1-q /a_1a_3  } \Big(A_1\sps A_3\Big) \bigg|_{a_1^0 a_3^0 }
\cr &= {1\over1-q a_1 }{1\over1-q /a_1} {1\over1-q a_1 a_3}
{1\over1-q /a_1a_3  }
 \bigg|_{a_1^0  }\Big(A_1\sps A_3\Big)\bigg|_{  a_3^0 }
. \label{4.47}
\end{align}
The last equality is due to the fact that $A_1$ and $A_3$ do not
contain $a_1$. Next we will compute  the constant term
$$
  Q'\ses {1\over1-q a_1 }{1\over1-q  /a_1} {1\over1-q a_1 a_3} {1\over1-q /a_1a_3  }
 \bigg|_{a_1^0  }.
$$
The surprise, which is the whole point of the factorization in
\eqref{4.42}, is that this leads to the same partial fraction
decomposition! More precisely we see that
$$
Q'\ses {B_1\over 1-qa_1 }\sps {B_3\over 1-qa_1 a_3}\ssp \bigg|_{
a_1^0  }\ses B_1+B_3
$$
with
\begin{align*}
 B_1 &= {a_1^2   a_3 \over   (a_1 -q)    (1-q a_1 a_3)
   (a_1 a_3-q)   }\ssp \bigg|_{ a_1=1/q  }
= {1\over (1-q^2)(1-a_3)(1-q^2/a_3)}\ses A_1 \\
B_3 &= {a_1^2   a_3 \over (1-q a_1 )   (a_1 -q)
   (a_1a_2a_3-q)   }\ssp  \bigg|_{ a_2=1/qa_1a_3 }
= { a_3\over (a_3-1) (1-q^2a_3)(1-q^2)} \ses A_3 .
\end{align*}
Thus \eqref{4.47} becomes
$$
G_3(q)\ses (A_1+A_3)^2\ssp \bigg|_{a_3^0}\ses A_1^2 \ssp
\bigg|_{a_3^0}\sps A_3^2\ssp \bigg|_{a_3^0}+2 A_1A_3\ssp
\bigg|_{a_3^0}.
$$
It is easy to see that the same collapse of terms occurs  in the
general case. Indeed we can rewrite \eqref{4.42} in the form
$$G_k(q)=
\bigg[\prod_{S\con[3,k]}{1\over 1- qa_1a_2A(S) }
\displaystyle{1\over 1- q/a_1a_2A(S)
}\bigg|_{a_2^0}\bigg]\bigg[\prod_{S\con[3,\ldots ,k]}{1\over 1-
qa_1A(S) } {1\over 1- q/a_1A(S)   } \bigg|_{a_1^0 } \bigg]\Bigg|_{
a_3^0\cdots a_k^0}. 
$$
We can see  that, in both constant terms with respect to $a_1$ and
$a_2$,     the first member of each pair of
 factors contributes and the second dually contributes, and  the partial fraction algorithm yields
$$
 \prod_{S\con[3,\ldots ,k]}{1\over 1- qa_1a_2A(S) }  {1\over 1- q/a_1a_2A(S)   }\bigg|_{a_2^0}
\ses \sum_{T\con[3,\ldots ,k]}  {C_T\over   1- qa_1a_2A(T)
}\bigg|_{a_2^0}\ses  \sum_{T\con[3,\ldots ,k]}   C_T
$$
with
\begin{align*}
  C_T
&= \Big(1- qa_1a_2A(T)  \Big)\prod_{S\con[3,\ldots ,k]}{1\over 1-
qa_1a_2A(S) }  {1\over 1- q/a_1a_2A(S)   } \Bigg|_{a_2=1/qa_1A(T)}
\cr &= {1\over \left(1- q/a_1a_2A(T)\right) }
\prod_{\multi{S\con[3,\ldots ,k]\cr S\neq T}}{1\over 1- qa_1a_2A(S)
}  {1\over 1- q/a_1a_2A(S) } \Bigg|_{a_2=1/qa_1A(T)} \cr &= {1\over
(1- q^2) } \prod_{\multi{S\con[3,\ldots ,k]\cr S\neq T}} {1\over 1-
A(S)/A(T) }  {1\over 1- q^2A(T)/ A(S)   }
\end{align*}
and we see that, as in the case $k=3$, all of these coefficients are
independent of $a_1$. Moreover we can also easily see that
$$
 \Big(1- qa_1 A(T) \Big)\prod_{S\con[3,\ldots ,k]}{1\over 1- qa_1 A(S) }  {1\over 1- q/a_1 A(S)   }
\Bigg|_{a_1=1/q A(T)}\ses C_T.
$$
This reduces the computation of $G_k(q)$ to the sum of
$2^{k-2}+{2^{k-2} \choose 2}$ constant terms of the form
$$
G_k(q)= \sum_{i=1}^{2^{k-2} } A_i^2\bigg|_{ a_3^0\cdots a_k^0} \sps
2 \sum_{1\le i<j\le 2^{k-2} }  A_iA_j\bigg|_{ a_3^0\cdots a_k^0}.
$$
Note that for $k=5$ we are reduced to the calculation of
$2^3+{2^3\choose 2}=36$ constant terms. Most importantly in each of
these constant terms the denominators have at most 14 factors. The
latest version of the partial fraction algorithm (whose update is
motivated by the computation of $G_5(q)$) posted in the web site

 \centerline {\ita
http://www.combinatorics.net.cn/homepage/xin/maple/ell2.rar}

\noindent computed these 36 constant terms on a Pentium 4 Windows
system computer with a 3G Hz processor in about 22 minutes which is
a considerable time reduction from the 2 hours and 15 minutes that
took  previous versions of the algorithm to compute these constant
terms.

The same approach can be used to calculate $W_5(q)$, but in a much
simpler way. The constant terms have to be appropriately modified.
Again we will start with the case $k=3$.

The $k$-tuple of sets interpretation of the constant  term in
\eqref{2.3} given in Section \ref{sec-2}, yields that to obtain
$W_k(q)$ we must compute the constant terms corresponding
 to the $2^k$ systems obtained by requiring each $A_i$ to have $2$ or $0$  more elements than its complement
in all possible ways and then carry out an inclusion exclusion type
alternating sum of the results.

 A moments reflection should reveal that  to get $W_3(q)=W_3(q,1)$ we need only modify \eqref{4.39} to
\begin{align}
W_3(q,t)&= \Big(\big(1-a_4/a_1\big) \big(1-1/a_2\big)
\big(1-1/a_3\big) \Big) \cr & \qquad \quad \times {1\over1-q
a_1a_2a_3} {1\over1-q a_1a_2 }
 {1\over1-q a_1 a_3}
{1\over1-q a_1 }\cr &\qquad\qquad\qquad \times
 {1\over1-q a_2a_3a_4}
{1\over1-q a_2a_4 } {1\over1-q a_3 a_4} {1\over1-q a_4 } {1\over
1-t/a_1a_2a_3a_4}\ssp \bigg|_{a_1^0a_2^0a_3^0a_4^0}.  \label{4.49}
\end{align}
In fact  expanding the first factor gives the 8 terms
$$
1-1/a_2-1/a_3-a_4/a_1+a_4/a_1a_2+a_4/a_1a_3+1/a_2a_3 -a_4/a_1a_2a_3.
$$
And we see that  the $8$ constant terms obtained by expanding this
factor in \eqref{4.49} correspond in order to the following $8$
modified versions of $\CS_3'$
$$
 \left \|
\mymatrix{ |A_1|&=& d\cr |A_2|&=& d \cr |A_3|&=& d \cr |\ ^cA_1|&=&
d\cr } \right. \scs \ess\ess\ess
 \left \|
\mymatrix{ |A_1|&=& d+1\cr |A_2|&=& d \cr |A_3|&=& d \cr |\
^cA_1|&=& d-1\cr } \right. \scs  \ess\ess\ess \left \| \mymatrix{
|A_1|&=& d\cr |A_2|&=& d+1 \cr |A_3|&=& d \cr |\ ^cA_1|&=& d\cr }
\right. \scs \ess\ess\ess \left \| \mymatrix{ |A_1|&=& d\cr |A_2|&=&
d \cr |A_3|&=& d+1 \cr |\ ^cA_1|&=& d\cr } \right.
$$
$$
 \left \|
\mymatrix{ |A_1|&=& d+1\cr |A_2|&=& d+1 \cr |A_3|&=& d \cr |\
^cA_1|&=& d-1\cr } \right. \scs \ess\ess\ess
 \left \|
\mymatrix{ |A_1|&=& d+1\cr |A_2|&=& d \cr |A_3|&=& d+1 \cr |\
^cA_1|&=& d-1\cr } \right. \scs  \ess\ess\ess \left \| \mymatrix{
|A_1|&=& d\cr |A_2|&=& d+1 \cr |A_3|&=& d+1 \cr |\ ^cA_1|&=& d\cr }
\right. \scs \ess\ess\ess \left \| \mymatrix{ |A_1|&=& d+1\cr
|A_2|&=& d+1 \cr |A_3|&=& d+1 \cr |\ ^cA_1|&=& d-1\cr } \right.
$$
\sas

\noindent Now the elimination of $a_4$ in \eqref{4.49} and then
setting $t=1$ (as for $G_3(q)$) gives \begin{align*}
 W_3(q)\ses
&\Big(\big(1-1/a_1^2a_2a_3\big) \big(1-1/a_2\big) \big(1-1/a_3\big)
\Big)\times \cr &\ess\ess\ess\times {1\over1-q a_1a_2a_3} {1\over1-q
a_1a_2 }
 {1\over1-q a_1 a_3}
{1\over1-q a_1 } \cr
&\ess\ess\ess\ess\ess\ess\ess\ess\ess\ess\ess\ess\times {1\over1-q
/a_1} {1\over1-q /a_1a_3 } {1\over1-q  /a_1a_2} {1\over1-q
/a_1a_2a_3}\ssp \bigg|_{a_1^0a_2^0a_3^0 }.
\end{align*}
For general $k$, we are left to compute the constant term
$$ W_k(q) =  \big(1-1/a_1^2a_2\cdots a_k\big)\prod_{i=2}^k \big(1-1/a_i \big)  \bigg(\prod_{S\con[2,k]}{1\over 1- qa_1A(S) }
 {1\over 1- q/a_1A(S)   }\bigg)
\bigg|_{a_1^0a_2^0\cdots a_k^0  }.  $$ Using this formula, the
updated package will directly deliver $W_5(q)$ in about 17 minutes.
This is because   the factors in the numerator nicely cancel some of
the denominators of the intermediate rational functions.

\bibliographystyle{plain}


\end{document}